\documentclass{amsart}
\usepackage{amssymb,latexsym}
\usepackage[all]{xy}
\usepackage{enumerate}
\nonstopmode
\theoremstyle{plain}
\newtheorem*{main*}{Main Theorem}
\newtheorem{theorem}{Theorem}
\newtheorem{corollary}[theorem]{Corollary}
\newtheorem{proposition}[theorem]{Proposition}
\newtheorem{lemma}[theorem]{Lemma}
\theoremstyle{definition}

\newtheorem{remark}[theorem]{Remark}

\newtheorem{claim}[theorem]{Claim}

\DeclareMathOperator{\hh}{h}
\DeclareMathOperator{\Res}{Res}\DeclareMathOperator{\len}{length}\newcommand{\I}{\mathcal{I}}
\newcommand{\A}{\sf{{A}}}
\newcommand{\B}{\sf{B}}
\newcommand{\C}{\sf{C}}  
\newcommand{\LL}{{\mathcal{L}}_3}
\newcommand{\NN}{{\mathbb{N}}}
\newcommand{\FF}{{\mathbb{F}}}
\newcommand{\QQ}{{\mathbb{Q}}}
\newcommand{\PP}{{\mathbb{P}}}
\newcommand{\KK}{{\mathbb{K}}}
\newcommand{\vdim}{{\mathrm{vdim}}}
\newcommand{\ch}{{\mathrm{char}}}

\usepackage[naturalnames=true]{hyperref}\date{}

\begin{document}
\title[Quintuple points]{Postulation of general 
quintuple \\fat point schemes in $\mathbb{P}^3$}

\author{E. Ballico}\address{Dept.\ of Mathematics\\ University of      
Trento\\38123 Povo (TN), 
Italy}\email{ballico@science.unitn.it}
\author{M. C. 
Brambilla}\address{Dip.\ di Scienze Matematiche\\
Universit\`a Politecnica delle Marche\\
I-60131 Ancona, Italy}
\email{brambilla@dipmat.univpm.it}
\author{F. Caruso}\address{Dept.\ of Mathematics\\ University of 
Trento\\38123 Povo (TN), Italy}\email{caruso@science.unitn.it}
\author{M. Sala}
\address{Dept.\ of Mathematics\\ University of Trento\\38123 Povo 
(TN), Italy}\email{sala@science.unitn.it}
\subjclass{14N05; 15A72; 65D05}
\keywords{polynomial interpolation; fat point; zero-dimensional 
scheme; projective space}

\begin{abstract}
We study the postulation of a general union $Y$ of double, triple, 
quartuple and quintuple points of $\mathbb {P}^3$.
In characteristic $0$, we prove
that $Y$ has good postulation in degree $d\ge 11$.
The proof is based on the combination of the Horace differential
lemma with a computer-assisted proof. 
We also classify the exceptions in degree $9$ and $10$.
\end{abstract}

\maketitle

\section{Introduction}

Let $\KK$ be a field of characteristic $0$, $n\in \NN$ and 
$\mathbb {P}^n=\PP^n(\KK)$.
In this paper we study the postulation of general fat point schemes 
of $\PP^3$ with multiplicity up to $5$.
A {\em fat point} $mP$ is a zero dimensional subscheme of $\mathbb {P}^3$
supported on a point $P$ and with $({\mathcal {I}}_{P,\mathbb {P}^3})^m$ as
its ideal sheaf. 
A {\em general fat point scheme} 
$Y=m_1P_1+\ldots+m_kP_k$, with
$m_1\ge\ldots\ge m_k\ge1$ is a general 
zero-dimensional scheme such
that its support $Y_{\textrm{red}}$ is a union of 
$k$ points and for each
$i$ the connected component of $Y$ supported 
on $P_i$ is the fat point
$m_iP_i$. 
We say that the multiplicity of 
$Y$
is the maximal multiplicity, $m_1$, of its components.

Studying the 
postulation of $Y$ means to compute the dimension of the space
of 
hypersurfaces of any degree containing the scheme $Y$. In other
words 
this problem is equivalent to
computing the dimension $\delta$ of the space of homogeneous 
polynomials of any degree
vanishing at each point $P_i$ and with all 
their derivatives, up to multiplicity $m_i-1$, vanishing at $P_i$.
We say that $Y$ has good postulation if $\delta$ is the
expected dimension, that is, 
either the difference between the dimension of the polynomial space 
and the number of imposed conditions or just the dimension of 
the polynomial space (when $\delta$ would exceed it).

This problem was investigated by many authors in the case of 
$\mathbb {P}^2$,
where we have the important Harbourne-Hirschowitz 
conjecture
(see \cite{ciliberto} for a survey).
In the case of 
$\mathbb {P}^n$, for $n\ge 2$, the celebrated Alexander-Hirschowitz 
theorem gives a
complete answer  in the case of double points, that 
is when $m_i=2$ for any $i$
(\cite{AHinv,AHjag}, for a survey 
see \cite{BO}).
For arbitrary multiplicities and arbitrary 
projective
varieties there is a beautiful asymptotic theorem 
by
Alexander and Hirschowitz \cite{AH}.

Here we focus on the case of general fat point schemes
$Y\subset\mathbb {P}^3$. In this case a general 
conjecture which characterizes all the
general fat point schemes not having good postulation was proposed 
by Laface and Ugaglia in \cite{laface-ugaglia}.
The good postulation of general fat point schemes of multiplicity 
$4$ was proved for degrees $d\ge 41$ in \cite{bb} by the first two authors. 
Then Dumnicki made a real breakthrough.
In particular he showed, in \cite{dumnicki2}, how to check the cases with degree  
$9 \le d \le 40$. 
Stimulated by his results, we consider now the case of fat point schemes 
of multiplicity $5$ and we solve completely the problem of the good postulation.
Indeed we prove the following theorem.
\begin{theorem}\label{i1}
Let $\PP^3=\PP^3(\KK)$, where $\KK$ is a field of characteristic $0$. 
Fix non-negative integers $d, 
w, x, y, z$ such that $d \geq 11$. Let $Y\subset \PP^3$ 
be a general union of $w$ 5-points, $x$ 4-points, $y$ 3-points and $z$ 
2-points. Then $Y$ has good postulation with respect to degree-$d$ 
forms.
\end{theorem}
%
The more natural way to prove our result would be to adopt a usual
two-parts proof: we might prove the theorem for $d \ge 66$ with the
same theoretical approach as in  \cite{bb} and then we might
prove the remaining finite cases with the computer.
We do not follow this consolidated path, because the
computer calculations at level $d\geq 60$ are infeasible
with nowadays means.
Instead, the proof of our result is an innovative combination of 
computer computation and theoretical argument, as in the following
logical outline:
\begin{itemize}
\item[a)] First, we prove  Theorem \ref{i1} for degrees $d=11$
          using our servers (Th. \ref{th11-21}).
\item[b)] Second, we improve the argument of \cite{bb} and so
      we are able to prove Theorem \ref{i1} for degrees $d\ge 53$, 
      with a theoretical proof depending on {\em both}
      known results (Remark \ref{y}) in the case of fat points of $\PP^2$ 
        {\em{and}} on a) ($d=11$). \\ This is presented in Section \ref{Sc}.
\item[c)] Then, we perform several computer calculations (Lemma \ref{punti-new}).
\item[d)] Then, we give a theoretical proof that restricts the
          required computations for the remaining cases ($11\leq d \leq 52$)
          to some feasible jobs. 
          This proof depends on the previous computational results. 
          The main point here is that 
          an iterated use of some results by Dumnicki 
          (\cite{dumnicki,dumnicki2}) allows us to greatly reduce the number 
          of cases 
          to be considered, by adding points of higher multiplicity.
          In particular we make use of points of multiplicity $10$ and $13$.
          Another tool we use is a result concerning low degrees and 
          few quintuple points 
          (see Proposition \ref{pochi5}). This result is proved by 
          a modification 
          of the general proof contained in Section \ref{Sc} and indeed allows 
          us to exclude many cases 
          from the explicit checking by computer. 
          All this is reported in Section \ref{intermedia}.
\item[e)] Finally, we perform direct computer checks for the surviving cases, as detailed in Section~\ref{computer}.
          Our computer calculations are deterministic and produce several digital certificates,
         that allow any other researcher to verify our results precisely. They rely on the efficient
         software package MAGMA (\cite{MAGMA}), whose linear algebra 
         over finite fields
         outperforms any other software that we tried.\\
         All our programmes and their digital certificates 
         are publicly accessible at
         \begin{verbatim} 
            http://www.science.unitn.it/~sala/fat_points/
         \end{verbatim}
\end{itemize}
In the remainder of the paper we provide two sections, as follows.

In Section \ref{d10} we classify all the exceptions arising in degree $9$ 
and $10$ (relying again on a computer-aided proof). 
 It turns out that, in these cases, the Laface-Ugaglia conjecture is true.

In Section \ref{fur} we collect several remarks on our results and their
consequences.


\section{Preliminaries}

In this section we fix our notation (which is the same as in \cite{bb} 
whenever possible),
prove several preliminary results and summarize our computational
results.

Let $\PP^n$ be the projective space on a field $\KK$, with
$\textrm{char}(\KK) = 0$ and $n\in \NN$.
Note that we do {\em not} assume that $\KK$ is algebraically closed.
However, some of the references which we will use assume 
that the base field is algebraically closed. In the next lemma
we explain why we are allowed to use these results.
\begin{lemma}
\label{www}
Let $\overline{\mathbb {K}}$ denote the algebraic closure of $\mathbb {K}$. 
Fix non-negative integers $n, d, x, y, z, w, s$ such that $n \ge 1$.
Assume that a general disjoint union of $w$ quintuple points, $x$ quartuple points, $y$ triple points, 
$z$ double points and $s$ (simple) points in $\mathbb {P}^n(\overline{\mathbb {K}})$
has good postulation in degree $d$, i.e. either
\begin{itemize}
\item $h^0(\mathbb {P}^n(\overline{\mathbb {K}}),\mathcal {I}_Z(d))=0$ and $\binom{n+4}{4}w+\binom{n+3}{3}x+\binom{n+2}{2}y+ (n+1)z+s\ge \binom{n+d}{n}$
\item
or $h^1(\mathbb {P}^n(\overline{\mathbb {K}}),\mathcal {I}_Z(d))=0$ and $\binom{n+4}{4}w+\binom{n+3}{3}x+\binom{n+2}{2}y+ (n+1)z+s\le \binom{n+d}{n}$.
\end{itemize} 
Then there is a disjoint union $W$ of $w$ quintuple points, $x$ quartuple points, $y$ triple points, $z$ double points and $s$ points in $\mathbb {P}^n(\mathbb {K})$
with good postulation in degree $d$.
\end{lemma}
\begin{proof}
Increasing $s$, if necessary, we reduce to the case 
$\binom{n+4}{4}w+\binom{n+3}{3}x+\binom{n+2}{2}y+ (n+1)z+s\ge \binom{n+d}{n}$. Let $\nu=w+x+y+z+s$.
Let $E$ be the subset of $\mathbb {P}^n(\overline{\mathbb {K}})^{\nu}$ parametrizing all the $\nu$-ples of 
distinct points of $\mathbb {P}^n(\overline{\mathbb {K}})$. 
For any $A\in E$, let $Z_A\subset \mathbb {P}^n(\overline{\mathbb {K}})$
be the fat point subcheme of $\mathbb {P}^n(\overline{\mathbb {K}})$ in which the first $w$ 
(resp. $x$, resp. $y$, resp. $z$, resp. $s$) fat points share multiplicity $5$ 
(resp. $4$, resp. $3$, resp. $2$, resp. $1$)
and $(Z_A)_{red}$ is the set associated to $A$. 
By semicontinuity there is a non-empty open
subset  $U$ of $E$ such that for all $A\in U(\overline{\mathbb {K}}) $ we have 
$\hh^0(\mathbb {P}^n(\overline{\mathbb {K}}),\mathcal {I}_{Z_A}(d)) =0$.
Since $\mathbb {K}$ is infinite, $\mathbb {K}^{n\nu}$ is dense in $\overline{\mathbb {K}}^{n\nu}$. 
Hence $\mathbb {P}^n(\mathbb {K})^{\nu}$ is Zariski
dense in $\mathbb {P}^n(\overline{\mathbb {K}})^{\nu}$. Thus, there is $B\in U(\mathbb {K})$ such that 
the scheme $Z_B$ satisfies $\hh^0(\mathbb {P}^n({\mathbb {K}}),\mathcal {I}_{Z_B}(d))=0$ and 
it is defined over $\mathbb {K}$.
\end{proof}
\noindent
From now on, $\mathbb {K}$ is {\em any} field with $\textrm{char}(\KK) = 0$ and 
$\mathbb {P}^n=\mathbb {P}^n({\mathbb {K}})$.

For any smooth $n$-dimensional connected variety $A$,
any $P\in A$ and any integer $m>0$, an {\em $m$-fat point of $A$} (or just
$m$-point) $\{mP,A\}$ is defined to be the $(m-1)$-th
infinitesimal neighborhood of $P$ in $A$, i.e.\ the closed subscheme
of $A$ with $(\mathcal {I}_{P,A})^m$ as its ideal sheaf.
As a consequence, $\{mP,A\}_{red} = \{P\}$ and the length of $\{mP,A\}$ is 
$\len(\{mP,A\})= \binom{n+m-1}{n}$.
To ease our notation, we will write $mP$ instead of $\{mP,A\}$ when the space $A$ is clear
from the context, and mostly we will have $A=\PP^n$ for $n=2,3$.

We call {\em general fat point scheme of $A$}
(or {\em general union} for short) any union $Y=m_1P_1+\ldots+m_k P_k$, with
$m_1\ge\ldots\ge m_k\ge1$, and $P_1,\ldots,P_k$ general points of
$\PP^n$. We denote $\deg(Y)=\sum \len(m_i P_i)$.\\
Given a positive integer $d$, we will say that a zero-dimensional
scheme $Y$ of $\mathbb {P}^n$ has {\em good postulation in degree 
$d$} if the following conditions hold:
\begin{enumerate}
\item[(a)]
if $\deg(Y) \le \binom{n+d}{n}$, then
$\hh^1(\mathbb {P}^n,\mathcal{I}_{Y}(d))=0,$
\item[(b)] if $\deg(Y) \ge\binom{n+d}{n}$, then
$\hh^0(\mathbb {P}^n,\mathcal {I}_{Y}(d))=0.$
\end{enumerate}

We will also use  the notation $\mathcal{L}_n(d;m_1,\ldots,m_k)$ for the linear system of 
hypersurfaces of degree $d$ in $\PP^n$ passing through a general union $Y=m_1P_1+\ldots+m_k P_k$.
The {\em virtual dimension} of $L=\mathcal{L}_n(d;m_1,\ldots,m_k)$ is 
$$\vdim(L)=\binom{n+d}{n}-\deg(Y)-1$$
and the dimension of the linear system always satisfies $\dim(L)\ge\vdim(L)$.
We say that $L$ is {\em special} if $\dim(L)>\max\{\vdim(L),-1\}$.
It is easy to see that a linear system $L$ is special if and only if
the corresponding general union does not have good postulation in degree $d$.
For more details we refer to \cite{ciliberto}.

\begin{remark}
\label{proofofcorollary}
Let $d_o \geq 2$.
Assume that $Y$ is any general fat point scheme in $\PP^n$ such that
$\deg(Y)\geq \binom{n+d_0}{n}$.
If we know that $Y$ has good postulation in degree $d\ge d_0$, we can claim
that $Y$ has good postulation in any degree, as follows.

For $d\ge d_0$, there is nothing to prove.

Since, for any $d\ge 1$, there is an injective map
$$
  H^0(\mathbb {P}^n,\I_Y(d-1))\hookrightarrow H^0(\mathbb {P}^n,\I_Y(d))\,,
$$ 
then $\hh^0(\mathbb {P}^n,\I_Y(d))=0$ implies
$\hh^0(\mathbb {P}^n,\I_Y(d-1))=0$. But $h^0(\mathbb {P}^n,\I_Y(d_0))=0$ and
so $\hh^0(\mathbb {P}^n,\I_Y(d))=0$ for any $d<d_0$, which
proves that $Y$ has good postulation.\\
Similarly, if $\hh^0(H,\mathcal {I}_{Y\cap H}(d_0)) =0$, then $\hh^0(H,\mathcal {I}_{Y\cap H}(d_0-1))=0$.
\end{remark}

The following general lemma will be useful in the sequel.
\begin{lemma}\label{2ago}
Let $\Sigma$ be an integral projective variety on $\overline{\KK}$ and 
let $\mathcal{L}$ be a linear system (not necessarly complete) of divisors on $\Sigma$. 
Fix an integer $m \geq 1$ and a general point $P\in \Sigma$.
Let $\mathcal{L}(-mP)$ be the sublinear system of $\mathcal{L}$ formed by all divisors with a point of multiplicity at
least $m$ at $P$.  Then we have
$$\dim (\mathcal{L}(-mP)) \le \max\{ \dim (\mathcal{L})-m,-1 \},$$
and, for any $1\leq k\le m$,
$$\dim (\mathcal{L}(-mP)) \le \max\{ \dim (\mathcal{L}(-kP))-(m-k),-1 \}.$$
\end{lemma}

\begin{proof}
The case $m=1$ is obvious. 
We assume by induction that
$$\dim (\mathcal{L}(-(m-1)P)) \le \max\{ \dim (\mathcal{L})-m+1,-1 \}.$$
By \cite[Proposition 2.3]{chiantini-ciliberto} 
it follows that 
$$
  \dim (\mathcal{L}(-mP)) \leq \max\{ \dim (\mathcal{L}(-(m-1)P))-1,-1\} \,,
$$ 
and so we get the desired inequality.
The proof of the second inequality is analogous.
\end{proof}

In the following lemma we show that in order to prove Theorem \ref{i1} 
for all quadruples $(w,x,y,z)$ of non-negative integers it is
sufficient to prove it only for a small set of quadruples $(w,x,y,z)$.

\begin{lemma}\label{2agoprop}
Fix an integer $d>0$. 
For any quadruple of non-negative integers $(w,x,y,z)$, let $Y(w,x,y,z)\subset \mathbb{P}^3$ 
denote a general union of $w$ 5-points, $x$ 4-points, $y$ 3-points and $z$ 
2-points. 
If $Y(w,x,y,z)$ has good postulation in degree $d$ for any quadruple $(w,x,y,z)$ such that
$$\binom{d+3}{3}-3 \le 35w+20x+10y+4z \le \binom{d+3}{3}+\Delta$$
where 
\begin{equation}\label{deltaaa}
\Delta=\left\{\begin{array}{ll}
13&\mbox{if $w>0$ and $x=y=z=0$,} \\
8&\mbox{if $x>0$ and $y=z=0$,}\\
4&\mbox{if $y>0$ and $z=0$,}\\
1&\mbox{if $z>0$}
\end{array}
\right.
\end{equation}
then any general quintuple fat point scheme has good postulation in degree $d$.
\end{lemma}

\begin{proof}
If a quadruple $(w,x,y,z)$ is such that
$35w+20x+10y+4z \le\binom{d+3}{3}-4$, then we want to prove that $\hh^1(\mathbb {P}^n,\mathcal {I}_{Y}(d))=0,$
where $Y=Y(w,x,y,z)$. 
Let $z'>0$ be the integer such that
$\binom{d+3}{3}-3\le35w+20x+10y+4z+4z' \le\binom{d+3}{3}$.
By hypothesis we know that $Y'=Y(w,x,y,z+z')$ has good postulation,
 that is, 
$\hh^1(\mathbb {P}^3,\mathcal {I}_{Y'}(d))=0$.
Since $Y\subset Y'$, then it is easy to see that 
$\hh^1(\mathbb {P}^3,\mathcal {I}_{Y}(d))\le
\hh^1(\mathbb {P}^3,\mathcal {I}_{Y'}(d))=0$, and so $Y$ has good postulation.

Now assume that 
$w>0$, $x=y=z=0$, and $35w \ge\binom{d+3}{3}+14$.
Let $Y$ be the corresponding general union of $w$ quintuple points.
This time we want to prove that
$\hh^0(\mathbb {P}^3,\mathcal{I}_{Y}(d))=0.$
Let $w'>0$ such that
$\binom{d+3}{3}-21\le35(w-w') \le\binom{d+3}{3}+13$.
Now we consider the following subcases:
\begin{itemize}
\item
if $35(w-w') \ge\binom{d+3}{3}$,
then we take the union  $Y'=Y(w-w',0,0,0)$ of $w-w'$ quintuple general points. Since we can assume $Y'\subset Y$, we immediately have
$\hh^0(\mathbb{P}^3,\mathcal {I}_{Y}(d))\leq
\hh^0(\mathbb {P}^3,\mathcal {I}_{Y'}(d))=0$, and so $Y$ has good
postulation.

\item
If $\binom{d+3}{3}-5\le 35(w-w')  \le \binom{d+3}{3}-1$,
we take the union  $Y'=Y(w-w',0,0,0)$ of $w-w'$ quintuple general points.
Since $Y$ contains at least one further quintuple point,
we can consider $Y''=Y(w-w'+1,0,0,0)$ and we can assume that
$Y'\subset Y''\subseteq Y$.
Note that $Y'$ has good postulation by hypotesis, and $\hh^0(\mathbb{P}^3,\mathcal {I}_{Y'}(d))\leq 5$. Hence by Lemma \ref{2ago} we have
$\hh^0(\mathbb{P}^3,\mathcal {I}_{Y''}(d))
\le \max\{\hh^0(\mathbb{P}^3,\mathcal {I}_{Y'}(d))-5,0\}=0$. Then we have that $Y''$ has good postulation, and consequently $Y$ has good postulation.

\item
If $\binom{d+3}{3}-12 \le 35(w-w')  \le \binom{d+3}{3}-6$, then we take
$Y'=Y(w-w',0,1,0)$, i.e.\ a general union of $w-w'$ quintuple points and one triple point. 
Now $\binom{d+3}{3}-2 \le \deg(Y')  \le \binom{d+3}{3}+4$ and by hypothesis $Y'$ has good postulation.
Since we can assume $Y'\subset Y$, by Lemma \ref{2ago} we have
$\hh^0(\mathbb{P}^3,\mathcal {I}_{Y}(d))
\le \max\{\hh^0(\mathbb{P}^3,\mathcal {I}_{Y'}(d))-2,0\}=0$, and so $Y$ has good postulation.
\item
If $\binom{d+3}{3}-21 \le 35(w-w')  \le \binom{d+3}{3}-11$, 
then we take $Y'=Y(w-w',1,0,0)$.
Now $\binom{d+3}{3}-2 \le \deg(Y')  \le \binom{d+3}{3}+9$ and by hypothesis $Y'$ has good postulation.
Since we can assume $Y'\subset Y$, by Lemma \ref{2ago} we have
$\hh^0(\mathbb{P}^3,\mathcal {I}_{Y}(d))
\le \max\{\hh^0(\mathbb{P}^3,\mathcal {I}_{Y'}(d))-1,0\}=0$, and so $Y$ has good postulation.
\end{itemize}

Assume now $x>0$, $y=z=0$ and $35w+20x \ge\binom{d+3}{3}+9$.
Let $Y=Y(w,x,0,0)$ be the corresponding general union and we want to prove that
$\hh^0(\mathbb {P}^3,\mathcal{I}_{Y}(d))=0.$

If $35w \ge\binom{d+3}{3}$, 
then $Y'=Y(w,0,0,0)$ has good postulation by the previous step and clearly it follows 
that $Y$ has good postulation.
Otherwise there exists $0<x'< x$ such that
$\binom{d+3}{3}-11 \le35w-20(x-x') \le\binom{d+3}{3}+8$.

Now we consider the following subcases:
\begin{itemize}
\item
If $35w-20(x-x') \ge \binom{d+3}{3}-4$, then
we take the union  $Y'=Y(w,x-x',0,0)$.
Since $Y$ contains at least one further quartuple point, by Lemma \ref{2ago} we have
$\hh^0(\mathbb{P}^3,\mathcal {I}_{Y}(d))
\le \max\{\hh^0(\mathbb{P}^3,\mathcal {I}_{Y'}(d))-4,0\}=0$, and so $Y$ has good postulation.
\item
If $\binom{d+3}{3}-6\le 35w+20(x-x')= \binom{d+3}{3}-5$, 
then we take $Y'=Y(w,x-x',0,1)$ and we have
$\binom{d+3}{3}-2 \le\deg(Y')  \le \binom{d+3}{3}-1$ and by hypothesis $Y'$ has good postulation.
Since we can assume $Y'\subset Y$, by Lemma \ref{2ago} we have
$\hh^0(\mathbb{P}^3,\mathcal {I}_{Y}(d))
\le \max\{\hh^0(\mathbb{P}^3,\mathcal {I}_{Y'}(d))-2,0\}=0$, and so $Y$ has good postulation.
\item
If $\binom{d+3}{3}-11 \le 35w+20(x-x')  \le \binom{d+3}{3}-7$, 
then we take $Y'=Y(w,x-x',1,0)$ and we have
$\binom{d+3}{3}-1 \le \deg(Y')  \le \binom{d+3}{3}+3$ and by hypothesis $Y'$ has good postulation.
Since we can assume $Y'\subset Y$, by Lemma \ref{2ago} we have
$\hh^0(\mathbb{P}^3,\mathcal {I}_{Y}(d))
\le \max\{\hh^0(\mathbb{P}^3,\mathcal {I}_{Y'}(d))-1,0\}=0$, and so $Y$ has good postulation.
\end{itemize}

Now assume $y>0$, $z=0$ and $35w+20x+10y \ge\binom{d+3}{3}+6$.
Let $Y=Y(w,x,y,0)$ be the corresponding general union and we want to prove that
$\hh^0(\mathbb {P}^3,\mathcal{I}_{Y}(d))=0.$

If $35w+20x \ge\binom{d+3}{3}$, 
then $Y'=Y(w,x,0,0)$ has good postulation by the previous steps and clearly it follows 
that $Y$ has good postulation.
Otherwise there exists $0<y'< y$ such that
$\binom{d+3}{3}-4 \le 35w-20x+10(y-y') \le\binom{d+3}{3}+5$.

Now we consider the following subcases:
\begin{itemize}
\item
If $35w-20x+10(y-y') \ge \binom{d+3}{3}-3$, then
we take the union  $Y'=Y(w,x,y-y',0)$.
Since $Y$ contains at least one further triple point, by Lemma \ref{2ago} we have
$\hh^0(\mathbb{P}^3,\mathcal {I}_{Y}(d))
\le \max\{\hh^0(\mathbb{P}^3,\mathcal {I}_{Y'}(d))-3,0\}=0$, and so $Y$ has good postulation.
\item
If $35w-20x+10(y-y') = \binom{d+3}{3}-4$, 
then we take $Y'=Y(w,x,y-y',1)$ and we have that
$\deg(Y')  = \binom{d+3}{3}$ and by hypothesis $Y'$ has good postulation.
It immediately follows that
$\hh^0(\mathbb{P}^3,\mathcal {I}_{Y}(d))
\le \hh^0(\mathbb{P}^3,\mathcal {I}_{Y'}(d))=0$, and so $Y$ has good postulation.
\end{itemize}

Finally assume that $z>0$ and $35w+20x+10y+4z \ge\binom{d+3}{3}+2$.
Let $Y=Y(w,x,y,z)$ be the corresponding general union and we want to prove that
$\hh^0(\mathbb {P}^3,\mathcal{I}_{Y}(d))=0.$

If $35w+20x+10y \ge\binom{d+3}{3}$, 
then $Y'=Y(w,x,y,0)$ has good postulation by the previous steps and clearly it follows 
that $Y$ has good postulation.
Otherwise there exists $0<z'< z$ such that
$\binom{d+3}{3}-2 \le 35w-20x+10y+4(z-z') \le\binom{d+3}{3}+1$.
Now we take the union  $Y'=Y(w,x,y,z-z')$.
Since $Y$ contains at least one further double point, by Lemma \ref{2ago} we have
$\hh^0(\mathbb{P}^3,\mathcal {I}_{Y}(d))
\le \max\{\hh^0(\mathbb{P}^3,\mathcal {I}_{Y'}(d))-2,0\}=0$, and so $Y$ has good postulation.
\end{proof}

\begin{remark}
Lemma \ref{2ago} and  Lemma \ref{2agoprop} heavily use 
$\mbox{char}(\mathbb {K})=0$, but they will be useful also
in Section \ref{computer}.
\end{remark}

Given a general fat point scheme $Y$ of $\mathbb {P}^n$ and a 
hyperplane $H\subset\mathbb {P}^n$, we will call {\em trace} of $Y$ the 
subscheme $(Y\cap H)\subset H$ and {\em residual} of $Y$ the scheme 
$\Res_H(Y)\subset\mathbb {P}^n$ with ideal sheaf 
$\I_Y:\mathcal{O}_{\mathbb {P}^n}(-H)$.
Notice that if $Y$ is an $m$-point supported on $H$, then its 
trace $Y\cap H$ is an $m$-point of $H$ and its residual $\Res_H(Y)$
is an $(m-1)$-point of $\mathbb {P}^n$. We will often use the following 
form of the so-called {\em  Horace
lemma}.

\begin{lemma}\label{castelnuovo}
Let $H\subset \mathbb {P}^n$ be a 
hyperplane and $X\subset \mathbb {P}^n$ a closed
subscheme. Then$$ \hh^0(\mathbb {P}^n,\I_X(d))\le\hh^0(\mathbb 
{P}^n,\I_{\Res_H(X)}(d-1))+\hh^0(H,\I_{X\cap H}(d))$$$$
\hh^1(\mathbb {P}^n,\I_X(d))\le\hh^1(\mathbb 
{P}^n,\I_{\Res_H(X)}(d-1))+\hh^1(H,\I_{X\cap 
H}(d))$$
\end{lemma}

\begin{proof}
The
statement is a straightforward consequence of the well-known
{\em 
Ca\-stel\-nuo\-vo exact sequence} $$0\to
\I_{\Res_H(X)}(d-1)\to \I_X(d)\to \I_{X\cap H}(d)\to 0.$$
For more details see e.g. \cite[Section 4]{BO}.
\end{proof}

The basic
tool we will need is the so-called {\em Horace differential lemma}. 
This technique allows us to take a {\em differential
trace} and a {\em differential residual}, instead of the classical 
ones. For an explanation of the geometric intuition of the
Horace differential lemma see \cite[Section 2.1]{AH}. Here we give 
only an idea of how the lemma works. Let $Y$ be an $m$-point
of $\mathbb {P}^n$ supported on a hyperplane$H\subset\mathbb {P}^n$. 
Following the language of Alexander and Hirschowitz, we can
describe $Y$ as formed by infinitesimally piling up some subschemes 
of $H$, called {\em layers}. For example the layers of a
$3$-point $\{3P,\mathbb {P}^n\}$ are $\{3P,H\}$,$\{2P,H\}$, and 
$\{P,H\}$. Then the differential trace can be any of these
layers and the differential residual is a {\em virtual} 
zero-dimensional scheme formed by the remaining layers. In this paper
we will apply several times the following result which is a 
particular case of the Horace differential lemma (see
\cite[Lemma 2.3]{AH}). 

\begin{lemma}[Alexander-Hirschowitz] \label{alehir} Fix an integer 
$m\ge 2$ and assume that
$\textrm{char}(\KK)=0$ or $\textrm{char}(\KK)>m$. Let 
$X$ be an $m$-point of $\mathbb {P}^n$ supported on $P$ and
$H\subset\mathbb {P}^n$ a hyperplane.Then for $i=0,1$ we have$$ 
\hh^i(\mathbb {P}^n,\I_X(d))\le \hh^i(\mathbb
{P}^n,\I_{R}(d-1))+\hh^i(H,\I_{T}(d))$$where the {\em differential 
residual } $R$ and the {\em differential trace} $T$ are
virtual schemes of the following type:

\begin{tabular}{|c|c|c|c|}
\hline
$m$& $T$&$R$&
\\\hline
$2$& $\{P,H\}$&$\{2P,H\}$&$(1,3)$
\\\hline
$3$& $\{P,H\}$&$(\{3P,H\},  \{2P,H\})$&$(1,6,3)$
\\\hline
$3$&$\{2P,H\}$ &$ (\{3P,H\},  \{P,H\})$ & $(3,6,1)$\\\hline
$4$& $\{P,H\}$&$( \{4P,H\}, \{3P,H\}, \{2P,H\})$&$(1,10,6,3)$
\\\hline$4$& $\{2P,H\}$&$(\{4P,H\}, \{3P,H\}, \{P,H\})$&$(3,10,6,1)$
\\\hline$4$& $\{3P,H\}$&$( \{4P,H\},  \{2P,H\}, \{P,H\})$&$(6,10,3,1)$
\\\hline$5$& $\{P,H\} $&$ (\{5P,H\}, \{4P,H\}, \{3P,H\}, 
\{2P,H\})$&$(1,15,10,6,3)$
\\\hline$5$& $\{2P,H\}$&$  (\{5P,H\}, \{4P,H\}, \{3P,H\}, 
\{P,H\})$&$(3,15,10,6,1)$
\\\hline$5$& $\{3P,H\}$&$(\{5P,H\}, \{4P,H\}, \{2P,H\}, \{P,H\})$&$(6,15,10,3,1)$
\\\hline$5$& $\{4P,H\}$&$  (\{5P,H\}, \{3P,H\}, \{2P,H\}, 
\{P,H\})$&$(10,15,6,3,1)$
\\\hline
\end{tabular}
\end{lemma}

In the previous lemma we described the 
possible differential residuals by writing the subsequent layers from 
which they are formed. These layers are obtained by intersecting with the 
hyperplane $H$ many times.
In particular the notation e.g.\ $R= 
(\{3P,H\},  \{2P,H\})$
means that $R\cap H= \{3P,H\}$ and 
$\mbox{Res}_H(R)\cap H= \{2P,H\}$, and, finally, 
$\mbox{Res}_H(\mbox{Res}_H(R))\cap H= \emptyset$, the latter equality 
being equivalent to 
$\mbox{Res}_H(\mbox{Res}_H(R))= \emptyset$, 
because $R_{red}\subset H$. Moreover, for each case in the statement 
we write in the last column the list of the lengths of the fat points 
of $H$ that we will obtain intersecting many times with $H$. 
Throughout the paper, when we will apply Lemma \ref{alehir}, we will 
specify which case we are considering by recalling this sequence of 
the lengths. For example, if we apply the first case of Lemma 
\ref{alehir}, we will say that we apply the lemma with respect to the 
sequence $(1,3)$.

The next two arithmeticals lemma will be used in the sequel.
\begin{lemma}\label{aritmetico}
Let $w,x,y,z$ be non negative integers such that 
$$ 35w+20x+10y+4z \le \binom{d+3}{3}+13.$$
Let $\alpha=\lfloor\frac{2x+y}{42}\rfloor$ and assume that $w\le\alpha-1$.
Then $35w\leq\frac1{12}\binom{d+3}3$.
\end{lemma}
\begin{proof}
By hypothesis we have
$20x+10y \le \binom{d+3}{3}+13$, from which we have
$$w\le\alpha-1\le\frac{20x+10y}{420}-1 \le\frac{1}{420} \binom{d+3}{3}+\frac{13}{420}-1\le\frac{1}{420} \binom{d+3}{3}.$$
\end{proof}

\begin{lemma}\label{c1}Fix 
non-negative integers $t, a,b,c,u,v, e,f,g,h$ such that $t \ge
18$, 
\begin{equation}\label{eqc1}15a +10b+6c +3u +v+10e+6f +3g +h\le 
\binom{t+2}{2}\end{equation} 
and $(e,f,g,h)$ is one of the
following quadruples: $(0,0,0,0)$,$(0,0,0,1)$, $(0,0,0,2)$, 
$(0,0,1,0)$, $(0,0,1,1)$, $(0,0,1,2)$, $(0,1,0,0)$,$(0,1,0,1)$,
$(0,1,0,2)$, $(0,1,1,0)$, $(1,0,0,0)$, $(1,0,0,1)$, $(1,0,0,2)$, 
$(1,0,1,0)$, $(1,0,1,1)$. Then the following inequality
holds:
\begin{align}\label{eqc2}  &10a+6b+3c+u+ 15e +15f + 15g +15h 
\le \binom{t+1}{2}.\end{align} 
If $e+f+g+h\le 2$, then the
statement holds for any $t\ge15$. If $e=f=g=h=0$, then the statement 
holds for any $t\ge 4$.\end{lemma}

\begin{proof} By
contradiction, let us assume that
\begin{equation}\label{prova}
10a+6b+3c+u+ 15e +15f + 15g +15h > \binom{t+1}{2},
\end{equation}
which, together with \eqref{eqc1}, implies 
\begin{equation}\label{altraprova}
5a +4b+3c +2u+v -5e-9f-12g-14h \le t+1.
\end{equation}
{}From \eqref{prova} and \eqref{altraprova} we get $$
\binom{t+1}{2}-2t-2 <-2b-3c-3u-2v+ 15(e+f+g+h)+ 2(5e+9f+12g+14h)
$$
that is
$$
\binom{t+1}{2}-2t-2 <-2b-3c-3u-2v+ 25e+33f+39g+43h<125,
$$
which implies
$t^2-3t-254<0$, which is false as soon as $t\ge18$.

If $e+f+g+h\le2$, the same steps give $t^2-3t-176<0$ which is false as soon as
$t\ge 15$.

If $e=f=g=h=0$, the same steps give $t^2-3t-4<0$ which is false as soon as
$t\ge 4$.
\end{proof}

\begin{remark}\label{pesante}
 Let $Y \subset \mathbb {P}^3$ be a 
zero-dimensional scheme and $H$ a hyperplane of $\mathbb {P}^3$. Fix 
non negative integers $c_2,c_3,c_4,c_5$. Denoting by $Y'$ the union of 
the connected components of $Y$ intersecting $H$, the scheme $Y 
\setminus Y'$ is a general union of
$c_5$ 5-points, $c_4$ 4-points, $c_3$ 3-points, and $c_2$ 
2-points. Moreover, the subscheme $Y'$ is supported on general points
of $H$ and it is given by a union of points of multiplicity $2$, $3$, 
$4$, $5$ or virtual schemes arising as residual in Lemma
\ref{alehir}. 
\end{remark}

In the following basic lemma we show 
how to apply the Horace differential Lemma \ref{alehir} in
our situation.

\begin{lemma}\label{lemma-a}
Fix a plane $H\subset \mathbb{P}^3$. Let $Y$
be a zero-dimensional scheme as in Remark \ref{pesante},
for some 
integers $c_2$, $c_3$, $c_4$, $c_5$. If the following
condition holds for some positive integer $t$:
\begin{equation}\label{condizione} 
\beta:=\binom{t+2}{2}-\deg(Y\cap H)\geq0,
\end{equation} 
then it is possible to degenerate $Y$ to a scheme $X$ such that
one of the following possibilities is verified:
\begin{enumerate}
\item[(I)] $\deg(X \cap H)=\binom{t+2}{2}$,
\item[(II)] $\deg(X \cap H)<\binom{t+2}{2}$, and all the irreducible
   components of $X$ are supported on $H$. This is possible only if
   $c_2+c_3+c_4+c_5<\beta$ and $c_2+c_3+c_4+c_5\le2$. 
\end{enumerate}
Moreover, if we assume $t\ge 18$ in case {\rm (I)} and $t\ge 15$ in 
case {\rm (II)},
we also have
\begin{equation}\label{caldo}
\deg(\Res_H(X)\cap H)\le \binom{t+1}{2}
\end{equation}
\end{lemma}

\begin{proof}
By specializing some of the connected components of $Y$ to isomorphic 
schemes supported on points
of $H$ we may assume that $\beta \ge 0$ is minimal. Let us denote now by $Y'$
the union of the connected components of $Y$ intersecting $H$.

By minimality of $\beta$ it follows that if $c_2>0$ then $\beta< 3$,
if $c_2=0$ and $c_3>0$ then $\beta< 6$, if $c_2=c_3=0$ and $c_4>0$ then
$\beta < 10$, if $c_2=c_3=c_4=0$ and $c_5>0$, then $\beta < 15$. If 
$c_2=c_3=c_4=c_5 =0$ and $\beta>0$, we are obviously in
case (II).

We degenerate now $Y$ to a scheme $X$ described as follows.
The scheme $X$ contains all the connected components of $Y'$.
Write
$$\beta= 10e+6f+3g+h$$
  for a unique quadruple of non-negative integers $(e,f,g,h)$ in the 
following list:
$(0,0,0,0)$,
$(0,0,0,1)$, $(0,0,0,2)$, $(0,0,1,0)$, $(0,0,1,1)$, $(0,0,1,2)$, $(0,1,0,0)$,
$(0,1,0,1)$, $(0,1,0,2)$, $(0,1,1,0)$, $(1,0,0,0)$, $(1,0,0,1)$, 
$(1,0,0,2)$, $(1,0,1,0)$, $(1,0,1,1)$
(i.e. in the list of Lemma \ref{c1}).
If $c_2>0$, then $e=f=g=0$ and $h\le 2$. If $c_2=0$ and $c_3>0$, then
$e=f=0$, $g \le 1$ and $h \le 2$. If $c_2=c_3=0$ and $c_4>0$, then 
$e=0$, $f\le 1$, $g \le 1$,
$h \le 2$ and $h=0$ if $f=g=1$.

Consider first the case $c_2>0$ and recall that in this case $e=f=g=0$
and $h\le2$. Assume now $c_2 \ge h$. Take as $X$ a general union
of $Y'$, $c_5$ 5-points, $c_4$ 4-points, $c_3$ 3-points, $(c_2-h)$ 2-points,
$h$ virtual schemes obtained by applying Lemma~\ref{alehir}
at $h$ general points of $H$ with respect to the sequence
$(1,3)$. Clearly we have $\deg(X\cap H)=\binom{t+2}{2}$.
Let us see now how to specialize $Y$ to $X$ in the remaining cases 
with $c_2>0$.
If $c_2 =1<h$ and $c_3+c_4+c_5\ge1$, then in the previous step we apply
Lemma \ref{alehir} using the unique 2-point and one 3-point or
4-point or 5-point with respect to the sequence
$(1,6,3)$ or $(1,10,6,3)$ or $(1,15,10,6,3)$
(recall that we assumed $c_i>0$ for at least one $i\in \{3,4,5\}$) 
and we conclude in the
same way.
If $c_2 =1<h$ and $c_3=c_4=c_5=0$, then we apply  Lemma \ref{alehir} to
the unique double point with respect to the sequence $(1,3)$, and we
are in case (II). Here and in all later instances of case (II) it is 
straightforward
to check that the inequalities $c_2+c_3+c_4+c_5<\beta$ and 
$c_2+c_3+c_4+c_5\le2$ are verified.

Assume now $c_2=0$ and $c_3>0$. Recall that $e=f=0$, $g \le 1$ and $h \le 2$.
If $c_3\ge g+h$ we take as $X$ a general union
of $Y'$, $c_5$ 5-points, $c_4 $ 4-points, $c_3-g-h$ $3$-points,
$g$ virtual schemes obtained applying Lemma \ref{alehir} at $f$
general points of $H$ with respect to
the sequence $(3,6,1)$ and $g$ virtual schemes obtained applying
Lemma \ref{alehir} at $g$ general points of $H$ with respect to the sequence
$(1,6,3)$.
If $0<c_3<g+h$ and $c_4+c_5\ge g+h-c_3$, then in the previous step we
apply Lemma \ref{alehir} using $c_3$ 3-points, and $(f+g-c_3)$ 
4-points or 5-points,
with respect to the sequences $(3,10,6,1)$ or $(1,10,6,3)$ or 
$(3,6,10,15,1)$ or $(1,15,10,6,3)$.
In all these cases we clearly have $\deg(X\cap H)=\binom{t+2}{2}$.
If  $c_2=0$, $0<c_3<g+h$ and $c_4+c_5< g+h-c_3$, then we have either
$c_3 =1$ and $c_4+c_5\le1$, or $c_3=2$, $g=1$, $h=2$ and $c_4=c_5=0$. 
In both cases
$\beta>c_3+c_4+c_5$. In this cases we can specialize all the components on
$H$, possibly applying Lemma \ref{alehir} and we are in case (II).

Now, assume that $c_2=c_3=0$ and $c_4>0$. Hence $e=0$, $f\le 1$, $g \le 1$,
$h \le 2$ and $h=0$ if $f=g=1$.
If $c_4+c_5\ge f+g+h$, then
we take as $X$ a general union
of $Y'$, $(c_4+c_5-f-g-h)$ 4-points or 5-points, $f$ virtual schemes 
obtained applying
Lemma \ref{alehir} at $f$ general points of $H$ with respect to the
sequence $(6,10,3,1)$ or $(6,15,10,3,1)$,
$g$ virtual schemes obtained applying Lemma \ref{alehir} at $g$
general points of $H$ with respect to the sequence $(3,10,6,1)$ or 
$(3,15,10,6,1)$ and
$h$ virtual schemes obtained applying Lemma \ref{alehir} at $h$
general points of $H$ with respect to the sequence $(1,10,6,3)$ or 
$(1,15,10,6,3)$.
Thus we have again $\deg(X\cap H)=\binom{t+2}{2}$.
If $c_2=c_3=0$ and  $0<c_4+c_5<f+g+h$, then we are again in case
(II), because we can specialize all the 4-points and 5-points on $H$
(possibly applying Lemma \ref{alehir}), since
$c_4+c_5\le f+g+h\le 3$ and $\beta = 6f+3g+h$; in this case we may also assume
that if $f\ne 0$ (i.e. $f=1$), then either one of the 4-points is specialized
with respect to the sequence $(6,10,3,1)$.

Finally assume that $c_2=c_3=c_4=0$ and $c_5>0$.
If $c_5 \ge e+f+g+h$ we apply Lemma \ref{alehir} at $e$ general 
points of $H$ with respect to the sequence $(10,15,6,3,1)$,
at $f$ general 
points of $H$ with respect to the sequence $(6,15,10,3,1)$,
at $g$ general 
points of $H$ with respect to the sequence $(3,15,10,6,1)$
and at $h$ general 
points of $H$ with respect to the sequence $(1,15,10,6,3)$.
In this way we arrive to case (I).
If $e+f+g+h > c_5$, then we start applying again Lemma \ref{alehir} 
as in the previous stop, but we have to stop at some point and we land in case (II).

Finally, we note that the property \eqref{caldo} follows immediately 
from the construction
above and from Lemma \ref{c1}.
\end{proof}

In order to prove the good postulation of schemes in $\mathbb{P}^3$ by 
applying induction, 
we need to know the good postulation of schemes in $\mathbb{P}^2$.
In the  next remark we point out the related results that we need.
\begin{remark}\label{y}
When the general union has multiplicity up to $4$ and $n=2$,
then we can use some results by Mignon (see \cite[Theorem 1]{m}). 
In particular we know that a general fat point scheme in $\PP^2$ of 
multiplicity $1\le m\le 4$ has good postulation in degree $d\ge 3m$. 
Interestingly, this result is valid for any characteristic 
of the ground field $\KK$ 
(for a discussion about $\ch(\KK)$ see Section \ref{fur}).\\
For multiplicities up to $7$ and when $\ch(\KK)=0$,
we can use some results by Yang (see \cite[Theorem 1 and Lemma 7]{y}), 
which imply that a general fat point scheme in $\PP^2$ of multiplicity $m\le 7$
has good postulation in degree $d\ge 3m$. 
\end{remark}

The case with no quintuple points has already been solved, as explained below.
\begin{remark}\label{yy}
When the general union $Y$ has multiplicity up to $4$ and $n=3$, we know
that $Y$ must have good postulation in any degree $d\geq 9$, thanks
to \cite{bb} and \cite{dumnicki2}. There is no self-contained theoretical proof for this, but we have 
a theoretical proof for $d\geq 41$ in \cite{bb}, along with a computer check up to $d=13$, and the missing 
computations can be found in \cite{dumnicki2}. 
\end{remark}


\subsection{Summary of our computational results}
\ \\
We list the results from Section \ref{computer} that we need
in the following sections.

\begin{lemma} \label{punti-new}
The following linear systems are non-special and have virtual 
dimension $-1$:
\begin{enumerate}
\item $\LL(3;2^5),$ 
\item $\LL(9;4^a,3^b)$ with $2a+b=22$, 
\item $\LL(9;5^4,4^4)$
\item $\LL(12;5^a,4^b,3^c)$ with $7a+4b+2c =91$.
\end{enumerate}
\end{lemma}

\begin{theorem}
\label{th11-21}
Fix non-negative integers $d, 
w, x, y, z$ such that $11\leq d \leq 21$ and $0\leq z\leq 4$.
Let $N= {{d+3}\choose{3}}$.
 Let $Y\subset
\mathbb{P}^3$  be a general union of $w$ 5-points, $x$ 4-points, $y$ 
3-points and $z$ 2-points such that
$$
  N-3 \leq 35w+20x+10y+4z \leq N+\Delta \,,
$$
where $\Delta$ is as in Lemma \ref{2agoprop}.

Then $Y$ has good postulation.
\end{theorem}
\begin{theorem}
\label{th22-37}
Fix non-negative integers $d,q, 
w, x, y, z$ such that:
\begin{itemize}
\item $22\leq d \leq 37$,
\item $0\leq z\leq 4$,
\item $0\leq 2x+y\leq 21$,
\item $0\leq w\leq 3$ or $0\leq x\leq3$.
\end{itemize}
Let $N= {{d+3}\choose{3}}$.
 Let $Y\subset
\mathbb{P}^3$  be a general union of $w$ 5-points, $x$ 4-points, $y$ 
3-points and $z$ 2-points such that
$$
  N-3 \leq 220q+35w+20x+10y+4z \leq N+\Delta \,,
$$
where $\Delta$ is as in Lemma \ref{2agoprop}.

Then $Y$ has good postulation.
\end{theorem}

\begin{theorem}
\label{th38-52}

Fix non-negative integers $d,r, 
w, x, y, z$ such that:
\begin{itemize}
\item $38\leq d \leq 52$,
\item $0\leq z\leq 4$,
\item $0\leq 2x+y\leq 41$,
\item $0\leq w\leq 12$.
\end{itemize}
Let $N= {{d+3}\choose{3}}$.
 Let $Y\subset
\mathbb{P}^3$  be a general union of $w$ 5-points, $x$ 4-points, $y$ 
3-points and $z$ 2-points such that
$$
  N-3 \leq 455r+35w+20x+10y+4z \leq N+\Delta \,,
$$
where $\Delta$ is as in Lemma \ref{2agoprop}.

Then $Y$ has good postulation.
\end{theorem}


\newpage

\section{The proof of Theorem \ref{i1} for high degrees }\label{Sc}

This section is devoted to the proof of Theorem \ref{i1} for high degrees, that is for $d\ge 53$.
Throughout the section we fix a hyperplane $H\subset\mathbb {P}^3$.
We recall that our ground field $\KK$ has characteristic zero.

In the different steps of the proof we will work with zero-dimensional
schemes that are slightly more general than a union of fat points. In
particular, we will say that a zero-dimensional
scheme $Y$ is {\em of type $(\star)$} if its irreducible components
are of the following type:
\begin{itemize}
\item[-]
$m$-points with $2\le m\le 5$, supported on general points of $\mathbb {P}^3$,
\item[-]
$m$-points with $1\le m\le 5$, or virtual schemes arising as residual
in the list of Lemma
\ref{alehir}, supported on general points of $H$.
\end{itemize}

Given a scheme $Y$ of type $(\star)$ satisfying
\eqref{condizione} for some integer $t$, we will say that $Y$ is of type
(I,$t$) if, when we apply Lemma \ref{lemma-a} to $Y$, we are in case
(I). Otherwise we say that $Y$ is of type (II,$t$).

We fix now (and we will use throughout this section)
the following notation, for any integer $t$:
given a scheme $Y_t$ of type $(\star)$ and satisfying
\eqref{condizione} for $t$, we will denote by $X_t$ the
specialization described in Lemma \ref{lemma-a}.
We write the residual $\Res_H(X_t)=Y_{t-1}\cup Z_{t-1}$,
where $Y_{t-1}$ is the union of all unreduced components of
$\Res_H(X_t)$ and $Z_{t-1}=\Res_H(X_t)\setminus Y_{t-1}$.
Clearly  $Z_{t-1}$ is the union of finitely many simple points of
$H$.
Thus, at each step $t \mapsto t-1$ we will have
$$
   Y_t \mapsto X_t \mapsto  \Res_H(X_t)= Y_{t-1}\cup Z_{t-1} \,.
$$
For any integer $t$, we set $z_{t}:= | Z_{t}|$,
$\alpha_{t}:=\deg(Y_{t})=\deg(X_{t})$, and
$$\delta_t:=\max\left(0,\binom{t+2}{3} -\deg(Y_{t-1}\cup
Z_{t-1})\right).$$

We fix the following statements:
\begin{itemize}
\item[-] ${\A}(t)=\{Y_t \mbox{ has good postulation in degree $t$}\}$,
\item[-] ${\B}(t)=\{\Res_H(X_t) \mbox{ has good postulation in degree 
$t-1$}\}$,
\item[-] ${\C}(t)=\{\hh^0(\mathbb {P}^3,\mathcal 
{I}_{\Res_H(Y_{t-1})}(t-2)) \le \delta_t\}$.
\end{itemize}

\begin{claim}\label{claim-a0}
Fix $t\ge 16$.
If $Y_t$ is a zero-dimensional scheme of type (II,$t$), then
it has good postulation, i.e. ${\A}(t)$ is true.  Moreover if $t\ge 
17$, also ${\B}(t)$ is true.
\end{claim}

\begin{proof}
Since $Y_t$ is of type (II,$t$), when we apply Lemma \ref{lemma-a} to 
$Y_t$, we obtain a
specialization $X_t$ such that all its irreducible components are supported on
$H$ and such that $\deg(X_t\cap H)\leq\binom{t+2}{2}$.

We prove now the vanishing $\hh^1({\bf{P}}^3,\mathcal {I}_{Y_t}(t)) = 0$.
By semicontinuity, it is enough to prove the
vanishing
$\hh^1({\bf{P}}^3,\mathcal {I}_{X_t}(t)) = 0$.
Notice that by taking the residual of $X_t$ with respect to $H$ for at most
five times we get at the end the empty set.

Since $\deg(X_t\cap H)\le\binom{t+2}{2}$ and $t\ge 15$,
by Remark \ref{y} it follows the vanishing
$\hh^1({\bf{P}}^2,\mathcal{I}_{X_t\cap H}(t)) = 0.$
Let $R_{t-1}$ denote the residual $\Res_H(X_t)$ and recall that
any component of $R_{t-1}$ is supported on $H$.
We check now that
$\hh^1({\bf{P}}^3,\mathcal{I}_{R_{t-1}}(t-1)) = 0.$
\\
In order to do this
we take again the trace and the residual with respect to $H$.
By \eqref{caldo} we know that
$\deg(\Res_H(X_t)\cap H)\leq\binom{t+1}{2}$
then again by Remark \ref{y}, since $t-1\ge15$, we have
$\hh^1({\bf{P}}^2,\mathcal{I}_{R_{t-1}\cap H}(t-1)) = 0.$

We repeat this step taking $R_{t-2}:=\Res_H(R_{t-1})$ and noting
that the trace $R_{t-2}\cap H$ has degree less or equal than
$\binom{t}{2}$, by Lemma \ref{c1}.
Moreover this time the scheme $R_{t-2}\cap H$ cannot contain quintuple
points,
in fact it is a general union of quartuple, triple, double and simple points.
Hence by Remark \ref{y} we have
$\hh^1({\bf{P}}^2,\mathcal{I}_{R_{t-2}\cap H}(t-2)) = 0$, since  $t-2\ge 12$.

We repeat once again the same step and we obtain
$R_{t-3}:=\Res_H(R_{t-2})$. Now the trace ${R_{t-3}\cap H}$ contains
only triple, double or simple points and so we have again the vanishing
$\hh^1({\bf{P}}^2,\mathcal{I}_{R_{t-3}\cap H}(t-3)) = 0$,
by Remark \ref{y}, since $t-3\geq 9$. Set $R_{t-4}:= \Res_H(R_{t-3})$.
The scheme $R_{t-4}\cap H$ is reduced and formed by less than 
$\binom{t-2}{2}$ general
points of $H$. Hence $\hh^1({\bf{P}}^2,\mathcal{I}_{R_{t-4}\cap H}(t-4)) = 0$.
Notice that this time the residual
$\Res_H(R_{t-4})$ must be empty and so, since
$\mathcal{I}_{\Res_H(R_{t-4})}=\mathcal{O}_{\mathbb {P}^3}$, we obviously have
$\hh^1({\bf{P}}^3,\mathcal{I}_{\Res_H(R_{t-4})}(t-5)) = 0.$
Hence  thanks to Lemma \ref{castelnuovo} we obtain
$\hh^1({\bf{P}}^3,\mathcal {I}_{Y_t}(t)) = 0$.

We also know that
\begin{equation}\label{stima}
\deg(Y_t)=\deg(X_t)\leq
\binom{t+2}{2}+\binom{t+1}{2}+\binom{t}{2}+\binom{t-1}{2}+\binom{t-2}{2} 
\leq \binom{t+3}{3}
\end{equation}
where the second inequality is equivalent to $\binom{t-2}{3}\geq0$, which is
true if $t\ge 4$. Hence $Y_t$ has good postulation,
that is, ${\A}(t)$ is true.

It is easy to see that  also the scheme $\Res(X_t)$ must be of type
(II,$t-1$). Hence ${\B}(t)$ follows from the
first part of the proof.
\end{proof}

\begin{claim}\label{claim-a}
Fix $t\ge15$.
If $Y_t$ is a zero-dimensional scheme of type (I,$t$), then ${\A}(t)$ is
true if ${\B}(t)$ is true.
\end{claim}

\begin{proof}
Since $Y_t$ is of type (I,$t$), we can apply Lemma \ref{lemma-a} and we obtain a
specialization $X_t$  such that $\deg(X_t\cap H)=\binom{t+2}{2}$.
Thus, since $t\geq 15$, by Remark \ref{y} it follows 
$$\hh^0(H,\mathcal {I}_{X_t\cap H}(t))=\hh^1(H,\mathcal {I}_{X_t\cap 
H}(t))=0.$$

Then, thanks to Lemma \ref{castelnuovo}, it follows, for $i=0,1$,
$$\hh^i(\mathbb {P}^3,\mathcal {I}_{X_t}(t))= \hh^i(\mathbb 
{P}^3,\mathcal {I}_{\Res_H(X_t)}(t-1)).$$
Thus in order to prove that the
scheme $X_t$ has good postulation in degree $t$,
it is sufficient to check the good postulation of $\Res_H(X_t)$ in degree
$t-1$.
\end{proof}
\begin{claim}\label{claim-b} If ${\A}(t-1)$ and ${\C}(t)$ are true, 
then ${\B}(t)$ is true.\end{claim}
\begin{proof} Recall that we write $\Res_H(X_t)=Y_{t-1}\cup 
Z_{t-1}$,where $Z_{t-1}$ is a union of simple points supported on 
$H$.
By \cite[Lemma 7]{bb},  to check that the scheme $\Res_H(X_t)$ 
has good postulation in degree $t-1$ (i.e.\ ${\B}(t)$), it is 
sufficient to check the good postulation of $Y_{t-1}$ in degree $t-1$ 
 (i.e.\ ${\A}(t-1)$) and to prove that ${\C}(t)$ is true.\end{proof}

\begin{claim}\label{claim-e}If $Y_t$ is of type (I,$t$), then ${\B}(t-1)$ 
implies ${\C}(t)$.\end{claim}

\begin{proof}The statement ${\C}(t)$ is 
true if $\hh^0(\mathbb {P}^3,\mathcal {I}_{\Res_H(Y_{t-1})}(t-2)) \le 
\delta_t.$ Note that since $\deg(X_t\cap H)=\binom{t+2}{2}$, we have$$ 
\deg(\Res_H(X_t))=\deg(Y_{t-1}\cup Z_{t-1})=\alpha_{t-1}+z_{t-1} 
=\alpha_t - \binom{t+2}{2},$$ and thus it 
follows $$\delta_t:=\max\left(0,\binom{t+2}{3} -\alpha_{t-1}-z_{t-1} 
\right)=\max\left(0,\binom{t+3}{3} -\alpha_{t}\right).$$ Notice that, 
by \eqref{caldo}, we have $\deg(\Res_H(X_{t})\cap H) \le 
\binom{t+1}{2}$. Hence, it 
follows $$\deg(\Res_H(Y_{t-1}))=\deg(\Res_H(\Res_H(X_{t})))\ge\alpha_{t}- 
\binom{t+2}{2}-\binom{t+1}{2}$$ and then, since 
$\binom{t+2}{2}+\binom{t+1}{2}=\binom{t+3}{3}-\binom{t+1}{3}$, we 
get $$ \deg(\Res_H(Y_{t-1}))\ge\binom{t+1}{3}-\binom{t+3}{3}+\alpha_t 
\ge \binom{t+1}{3}-\delta_t.$$ So in order to prove ${\C}(t)$ it is 
enough to prove that $\Res_H(Y_{t-1})$ has good postulation in degree 
$t-2$.\end{proof}

Now we are in position to prove our main result. 

\begin{proof}[Proof of  Theorem \ref{i1} for $d\ge53$] 
For all non negative integers $d\ge 53$ and $w,x,y,z$, we set 
$$\epsilon (d,w,x,y,z):= \binom{d+3}{3} - 35w -20x-10y-4z.$$ 
We will often write $\epsilon$ instead of $\epsilon (d,w,x,y,z)$
in any single step of the proofs in which the parameters 
$d,w,x,y,z$ are fixed. 

By Lemma \ref{2agoprop}, 
in order to prove 
our statement for all quadruples $(w,x,y,z)$ it is sufficient to 
check it for all quadruples $(w,x,y,z)$ such that $-13 \le \epsilon 
(d,w,x,y,z) \le 3$. We fix any such quadruple and we consider a general union $Y$ of $w$ 
5-points, $x$ 4-points, $y$ 3-points and $z$ 2-points. 

Notice that 
\begin{equation}\label{eqi0}w+x+y+z \ge \left\lceil 
\frac{1}{35}\left(\binom{d+3}{3}-3\right)\right\rceil\ge 
\frac{1}{35}\binom{d+3}{3}-\frac{3}{35} ,\end{equation} 
i.e. the 
scheme $Y$ has at least $\left\lceil 
\frac{1}{35}\left(\binom{d+3}{3}-3\right)\right\rceil$ connected 
components. 

The proof is by induction, based on Lemma 
\ref{castelnuovo}, 
and it requires different steps. 

Set $Y_d=Y$ and fix a hyperplane $H\subset\mathbb{P}^3$. 
 We can assume by generality that 
$\deg(Y_d\cap H)\le\binom{d+2}{2}$, hence we can apply Lemma \ref{lemma-a}, 
thus specializing the scheme $Y_d$ to a scheme 
$X_d$. If $Y_d$ is of type (II,$d$), then we conclude by Claim \ref{claim-a0}, 
since $d\ge 16$. Hence  we can assume that $Y_d$ is of 
type (I,$d$), and so, since $d\ge15$, 
by Claim \ref{claim-a} 
it is enough to check that the  
scheme $\Res_H(X_d)$ has good postulation in degree $d-1$. Now we write  
$\Res_H(X_d)=Y_{d-1}\cup Z_{d-1}$, where $Y_{d-1}$ is the union of all  
unreduced components of $\Res_H(X_d)$ and  
$Z_{d-1}=\Res_H(X_d)\setminus Y_{d-1}$. By Claim \ref{claim-b}, 
it is  
enough to prove that ${\A}(d-1)$ and ${\C}(d)$ are true. Notice that,  
since $d\ge18$, we get \eqref{caldo}, 
i.e.\ $\deg(Y_{d-1}\cap H)\leq 
\deg(\Res_H(X_d)\cap H)\le \binom{d+1}{2}$. Hence $Y_{d-1}$ satisfies 
condition \eqref{condizione} in degree $d-1$, then we can apply again  
Lemma \ref{lemma-a}. 

We have now two alternatives: either $Y_{d-1}$ is 
of type (I,$d-1$), or of type (II,$d-1$). 
In both cases, we note that by Claim \ref{claim-e} 
the statement ${\C}(d)$ follows from ${\B}(d-1)$, since 
$Y_d$ is of type (I,$d$). Now assume that $Y_{d-1}$ is of type (II,$d-1$). Then 
by Claim \ref{claim-a0}, 
since $d-1\ge17 $ we know that ${\B}(d-1)$ 
and ${\A}(d-1)$ are true and this concludes the proof. It remains to 
consider the case $Y_{d-1}$ of type (I,$d-1$). We apply again Claim \ref{claim-b} 
and we go on iterating the same steps.

\smallskip

Now we have two cases: either in a finite number $v$ of steps the
procedure described above gives us a scheme $X_{d-v}$ of type
(II,$d-v$), for a degree $d-v \ge 18$, or  the procedure goes on
until we get $X_{18}$, a scheme of type (I,$18$).    

In the first case, the steps of the procedure above prove that the
scheme $X_d$ has good postulation, and the statement is proved. 

Assume now that we are in the second case, i.e.\ $X_{18}$ is of type
(I,$18$), that is, $\deg(X_{18}\cap H)=\binom{20}2$. 
Note that, since $\epsilon\ge -13$, we have
\begin{equation}
\label{alfa}
\deg(X_{18})=\binom{21}{3}-\epsilon-\sum_{t=18}^{d-1} z_{t}\le \binom{21}3+13-\sum_{t=18}^{d-1}z_t.
\end{equation}
Now we want to estimate $\sum_{t=18}^{d-1} z_{t}$, 
which is the number of simple points we have removed in the steps above. 
Since we started from the scheme $Y_d$, in $d-18$ steps 
we arrived at the scheme $X_{18}$ in such a way that the case (II) never occurred.
Assume that in these $d-18$ steps we have applied $\gamma$ times Lemma \ref{alehir} 
with respect to sequences of type $(1,15,10,6,3)$, $(1,10,6,3)$, $(1,6,3)$ or $(1,3)$. 
As it is clear looking at the proof of Lemma \ref{lemma-a},
at each step the number of times we used a sequence giving as a trace a simple point is at most $2$, 
hence we have $\gamma\le 2(d-18)$.
Let $u_{18}$ denote the number of connected components of $X_{18}$. 
Hence it follows that 
\begin{equation}
\label{stima2}\sum_{t=18}^{d-1} z_{t}\ge 
w+x+y+z-2(d-18)-u_{18}.\end{equation}  
Now we need to estimate the number $u_{18}$. Let us denote by $T$ the union of components of $X_{18}$
of length $3$. Then any component of the scheme $X_{18}\setminus T$ has  length at least $4$, and $\deg(T)\leq\binom{20}2$ since the scheme $T$ is completely contained in the trace $X_{18}\cap H$.
So we have 
$$u_{18} \le \frac13\deg(T)+\frac14(\deg(X_{18})-\deg(T))\le\frac1{12}\binom{20}2+\frac14(\deg(X_{18})),$$
and using \eqref{stima2} and \eqref{eqi0}, we get
\begin{equation}\label{stimasemplici}
\sum_{t=18}^{d-1}z_t \ge \frac1{35}\binom{d+3}3-\frac3{35}-2(d-18)-\frac1{12}\binom{20}2-\frac14(\deg(X_{18}))).
\end{equation}
By using \eqref{alfa} and 
\eqref{stimasemplici}
we get 
$$
\frac34\deg(X_{18})\le \binom{21}3+13-\frac1{35}\binom{d+3}3+\frac3{35}+2(d-18)+\frac1{12}\binom{20}2
$$
and, since $d\ge 53$, it is easy to check that $\deg(X_{18})\leq \binom{21}3$. 
Note that $X_{18}$ depends implicitely on $d$, but we use the above inequality
to show what happend for $d=53$. Of cousre, for higher $d$'s we can easily show
that $\deg(X_{18})$ is actually even smaller, but we do not need it and 
we content ourselves with the claimed $\deg(X_{18})\leq \binom{21}3$

Hence we need to prove the vanishing $\hh^1(\mathbb {P}^3,\mathcal {I}_{X_{18}}(18))=0$, and by Claim 12,  it is enough to prove that
$\hh^1(\mathbb {P}^3,\mathcal {I}_{\Res_H(X_{18})}(17))=0$.

Now we change the procedure. Denote 
$\Res_H(X_{18})=R_{17}$. 

Since $\deg(R_{17}\cap H)\le \binom{19}2$, specializing some points on $H$ we can degenerate (without applying the Horace differential lemma) the scheme $R_{17}$ to a scheme $\widetilde{R}_{17}$ in such a way that one of the following cases happens:

(1a) either $\deg( \widetilde{R}_{17}\cap H)\leq \binom{19}2-15$ and all the components of $\widetilde{R}_{17}$ are supported on $H$,

(1b) or $\binom{19}2-14 \leq\deg( \widetilde{R}_{17}\cap H)\leq \binom{19}2$. 


Denote now, in both cases, $\widetilde{R}_{17}=E_{17}\cup F_{17}$, where $E_{17}$ is supported on $H$ and $F_{17}$ is supported outside $H$.
Take now the trace $\widetilde{R}_{17}\cap H=E_{17} \cap H$ and the residual $\Res_H(\widetilde{R}_{17})=R_{16}$.

Note that in case (1a) $F_{17}=\emptyset$,
while in  case (1b) 
$$\deg(R_{16}) \le \deg(X_{18}) - \binom{20}2 - \left(\binom{19}{2}-14\right).$$

By Lemma \ref{castelnuovo}, to prove 
that
$\hh^1(\mathbb {P}^3,\mathcal {I}_{R_{17}}(17))=0$,
it is enough to prove that
$\hh^1(\mathbb {P}^2,\mathcal{I}_{E_{17}\cap H}(17))=0$ 
(by Remark \ref{y}, since $17\ge 15$)
and
$\hh^1(\mathbb {P}^3,\mathcal{I}_{R_{16}}(16))=0.$

Now  we repeat the same step, that is, we specialize some points on $H$ without applying the Horace differential lemma, degenerating  $R_{16}$ to $\widetilde{R}_{16}$ in such a way that one of the following cases happens:

(2a) either $\deg( \widetilde{R}_{16}\cap H)\leq \binom{18}2-15$ and all the components of $\widetilde{R}_{16}$ are supported on $H$,

(2b) or $\binom{18}2-14 \leq\deg( \widetilde{R}_{16}\cap H)\leq \binom{18}2$. 

Denote again, in both cases, $\widetilde{R}_{16}=E_{16}\cup F_{16}$, where
$E_{16}$ is supported on $H$ and $F_{16}$ is supported outside $H$.

Note that $E_{16}$ is given by quintuple, quartuple, triple, double and simple points or virtual schemes arised by the application of Lemma 10. 
In any case taking the residual with respect to $H$ five times we get that the last residual has no components supported on $H$.

Take now the trace $\widetilde{R}_{16}\cap H=E_{16} \cap H$ and the residual $\Res_H(\widetilde{R}_{16})=R_{15}$. 

Note that in case (2a) $F_{16}=\emptyset$,
while in  case (2b) 
$$\deg(R_{15}) \le \deg(X_{18}) - \binom{20}2 - \left(\binom{19}{2}-14\right)- \left(\binom{18}{2}-14\right).$$

By Lemma \ref{castelnuovo}, to prove 
$\hh^1(\mathbb {P}^3,\mathcal {I}_{R_{16}}(16))=0$,
we only need
$\hh^1(\mathbb {P}^2,\mathcal{I}_{E_{16}\cap H}(16))=0$ 
(which is true by Remark \ref{y}, since $16\ge15$)
and
$\hh^1(\mathbb {P}^3,\mathcal{I}_{R_{15}}(15))=0.$

Now, without specializing furtherly, we denote 
$R_{15}=E_{15}\cup F_{15}$, where $E_{15}$ is supported on $H$ and $F_{15}=F_{16}$ is supported outside $H$.
Take now the trace $\widetilde{R}_{15}\cap H=E_{15} \cap H$ and the residual $\Res_H(\widetilde{R}_{15})=R_{14}$. 

By Lemma \ref{castelnuovo}, to prove 
that
$\hh^1(\mathbb {P}^3,\mathcal {I}_{R_{15}}(15))=0$,
it is enough to prove that
$\hh^1(\mathbb {P}^2,\mathcal{I}_{E_{15}\cap H}(15))=0$ (which is true by 
Remark \ref{y}, since $15\ge12$ and the trace contains at most quartuple points)
and $\hh^1(\mathbb {P}^3,\mathcal{I}_{R_{14}}(14))=0.$

We repeat again the same step and we get ${R}_{14}=E_{14}\cup F_{14}$, where $E_{14}$ is supported on $H$ and $F_{14}=F_{16}$ is supported outside $H$.

Take now the trace $\widetilde{R}_{14}\cap H=E_{14} \cap H$ and the residual $\Res_H(\widetilde{R}_{14})=R_{13}$.

By Lemma \ref{castelnuovo}, to prove 
that
$\hh^1(\mathbb {P}^3,\mathcal {I}_{R_{14}}(14))=0$,
it is enough to prove that
$\hh^1(\mathbb {P}^2,\mathcal{I}_{E_{14}\cap H}(14))=0$ 
(which is true by Remark \ref{y}, 
since the trace contains at most triple points)
and
$\hh^1(\mathbb {P}^3,\mathcal{I}_{R_{13}}(13))=0.$

We repeat again the same step and we get
${R}_{13}=E_{13}\cup F_{13}$, where $E_{13}$ is supported on $H$ and $F_{13}=F_{16}$ is supported outside $H$. 

Take now the trace $\widetilde{R}_{13}\cap H=E_{13} \cap H$ and the residual $\Res_H(\widetilde{R}_{13})=R_{12}$.
 
By Lemma \ref{castelnuovo}, to prove 
that
$\hh^1(\mathbb {P}^3,\mathcal {I}_{R_{13}}(13))=0$,
it is enough to prove that
$\hh^1(\mathbb {P}^2,\mathcal{I}_{E_{13}\cap H}(13))=0$ 
(which is true by Remark \ref{y}, since the trace contains at most double points)
and
$\hh^1(\mathbb {P}^3,\mathcal{I}_{R_{12}}(12))=0.$

Now we take again for the last time the trace and the residual with respect to $H$.
Denote 
${R}_{12}=E_{12}\cup F_{12}$, where $E_{12}$ is supported on $H$ and $F_{12}=F_{16}$ is supported outside $H$.
Taking the trace and the residual we have that
 $E_{12}\cap H$ is given by general simple points in $H$ and obviously we have
 $\hh^1(\mathbb {P}^2,\mathcal{I}_{E_{12}\cap H}(12))=0$. 
  
So we need only to show that the residual 
$\Res_H(E_{12}\cup F_{12})=R_{11}$
satisfies
$\hh^1(\mathbb {P}^3,\mathcal{I}_{R_{11}}(11))=0$.

Note that the residual does not have components supported on $H$.
More precisely, $R_{11}= F_{12}=F_{16}$, that is, the residual is a general 
collection of double, triple, quartuple and quintuple points.

But Theorem \ref{th11-21} ensures that any general collection of double, triple, quartuple and quintuple points has good postulation in degree $11$.

So in order to conclude the proof of the theorem it is enough to prove the following inequality:
\begin{equation}\label{quellochevoglio}
\deg(R_{11})=\deg(F_{16})\le \binom{14}3.
\end{equation}

Let us check this condition in any of the previous cases:
in cases (1a) and (2a) we have $F_{16}=\emptyset$ and so condition \eqref{quellochevoglio} is obviously satisfied.  
It remains to prove \eqref{quellochevoglio}
in case (2b), where
$$\deg(R_{11})\leq \deg(R_{15}) \le \deg(X_{18}) - \binom{20}2 - \left(\binom{19}{2}-14\right)- \left(\binom{18}{2}-14\right).$$

By \eqref{alfa} and \eqref{stimasemplici} we have
$$ \deg(X_{18})\le
\binom{21}3+13-
\left(\frac1{35}\binom{d+3}3-\frac3{35}-2(d-18)-\frac1{12}\binom{20}2-\frac14(\deg(X_{18})))\right)
$$
from which we obtain
$$\deg(X_{18})\le\frac43\left(
\binom{21}3+13-\frac1{35}\binom{d+3}3+\frac3{35}+2(d-18)+\frac1{12}\binom{20}2\right)
$$
and so
$$ \deg(R_{15})\le \deg(X_{18}) - \binom{20}2 - \left(\binom{19}{2}-14\right)- \left(\binom{18}{2}-14\right).$$

It is easy to check that, for any $d\ge53$ the inequality 
$$\frac43\left(
\binom{21}3+13-\frac1{35}\binom{d+3}3+\frac3{35}+2(d-18)+\frac1{12}\binom{20}2\right)
 - \binom{20}2 - \left(\binom{19}{2}-14\right)$$
 $$- \left(\binom{18}{2}-14\right)\le\binom{14}3
$$
is verified, and this implies \eqref{quellochevoglio} and completes the proof.
\end{proof}

\newpage

\section{The proof of Theorem \ref{i1} for low degrees}\label{intermedia}

In this section we discuss Theorem \ref{i1} in the remaining cases, 
that is, when the degree $d$ satisfies $11\le d\le 52$. 
In these cases the proof is based on computer calculations,
which are described explicitly in Section \ref{computer}.
Although in principle it is possible to go through all cases
in Lemma \ref{2agoprop} for $11\leq d\leq 52$, this is impractical
with nowadays computers.
In order to shorten the computational time
we need some other auxiliary theoretical results, that we develop in this section.
First we prove Theorem \ref{i1} for degrees $\ge 38$ in the special case
when we have few quintuple points (more precisely when $35w\le \frac1{12}\binom{d+3}{3}$).
Then we will present how to apply a result by Dumnicki in order 
to greatly reduce the cases to be tested by our computers.
\\

The proof of the following proposition is a modification of the argument in the previous section,
where we proved  Theorem \ref{i1} for $d\geq 53$.
\begin{proposition}\label{pochi5} 
Fix non-negative integers $d\ge 38 $, 
$w, x, y, z$ such that $35w\le \frac1{12}\binom{d+3}{3}$ and 
$$\binom{d+3}{3}-3 \le 35w+20x+10y+4z \le \binom{d+3}{3}+13.$$
Let $Y\subset \mathbb{P}^3$ 
be a general union of $w$ 5-points, $x$ 4-points, $y$ 3-points and $z$ 
2-points. Then $Y$ has good postulation in degree $d$.
\end{proposition}
\begin{proof}
We follows the same procedure as in the main proof of Section \ref{Sc}.
The first difference is that every time we apply
Lemma \ref{lemma-a}, we specialize on the plane $H$ as many quintuple points as possible. 

So starting with $Y_d=Y$, we obtain in a finite number of steps $d-d_0$ a scheme $X_{d_0}$
which does not contain quintuple points. We prove that in particular $d_0\ge 20$. 
Indeed
since by assumption we have
$35w\le \frac1{12}\binom{d+3}{3}$ and $d\ge 20$, it easily follows that 
$w\le \frac1{35}\left(\binom{d+3}{3}-\binom{22}{3}\right)$.  

Thus we have a general union $X_{d_0}$ of quartuple, triple and double points, 
and of virtual schemes of the type listed in the table of Lemma \ref{alehir}, 
arised by the application of Lemma \ref{lemma-a}.
  
If $X_{d_0}$ is of type (II,$d_0$), then we conclude, as in the previous proof, 
that it has good postulation and this implies that $Y$ has good postulation. 

Let us assume that $X_{d_0}$ is of type (I,$d_0$). 
Applying again Lemma \ref{lemma-a}
we can go on with our usual argument and we will obtain or a scheme of type (II,$e$), 
for some $e\ge18$, which 
concludes the proof, or a scheme $X_{18}$ of type (I,$18$) and without quintuple points.

At this point we can apply the same argument used in the proof of
\cite[Theorem 1] {bb}, regarding union of quartuple, triple and double 
points and virtual schemes of the type listed in \cite[Lemma 4] {bb}. 
In particular we apply \cite[Lemma 8] {bb} until we get a scheme $X_e$ 
of type (II,$e$) in degree $e\ge 13$. In this case we conclude, as in \cite{bb},
that our scheme has good postulation.

Now it remains to consider the case when we get a scheme $X_{13}$ of type (I,$13$).
Notice that in this case we want to prove that 
$\hh^1({\bf{P}}^3,\mathcal {I}_{X_{13}}(13)) = 0$.

Indeed let us prove that $\deg(X_{13})\le\binom{16}{3}$.
First of all, note that, since $35w\le\frac1{12}\binom{d+3}{3}$,
\begin{equation}\label{numerocomp}
w+x+y+z \ge x+y+z \ge \frac{1}{20}\cdot\frac{11}{12}\binom{d+3}{3}-\frac3{20}
\end{equation}
and setting $\epsilon= \binom{d+3}{3} - 35w -20x-10y-4z\geq-13$ we have:
\begin{equation}\label{degghe}
\deg(X_{13})=\binom{16}{3}-\epsilon-\sum_{t=13}^{d-1} z_{t}\le \binom{16}3+13-\sum_{t=13}^{d-1}z_t,
\end{equation}
where $z_t$ denotes, as in Section \ref{Sc}, the number of simple points we have removed
at the $(d-t)$-th step.
As in \eqref{stima2} we have
$$\sum_{t=13}^{d-1} z_{t}\ge w+x+y+z-2(d-13)-u_{13},$$
where $u_{13}$ is the number of connected component of $X_{13}$. 
Since $X_{13}$ does not contain simple points we have $u_{13}\le\frac13\deg(X_{13})$
and so by \eqref{numerocomp} we get
\begin{equation}\label{zzeta}
\sum_{t=13}^{d-1}z_t \ge 
\frac{11}{240}\binom{d+3}{3}-\frac3{20}-2(d-13)-\frac13(\deg(X_{13}))
\end{equation}
and by \eqref{degghe} we get
\begin{equation}\label{dacitare}
\deg(X_{13})\le \frac32\left(\binom{16}3+13-
\frac{11}{240}\binom{d+3}{3}+\frac3{20}+2(d-13)\right)
\end{equation}
But now it is easy to check that
$$\frac32\left(\binom{16}3+13-
\frac{11}{240}\binom{d+3}{3}+\frac3{20}+2(d-13)\right)\le \binom{16}{3}$$
as soon as $d\ge 30$.
Then it is enough to prove that
$\hh^1({\bf{P}}^3,\mathcal {I}_{X_{13}}(13)) = 0$.

Now we apply the residual without specializing any further components on $H$. 
In other words we take $Y_{12}:= \mbox{Res}_H(X_{13})$, 
$Y_{11}:= \mbox{Res}_H(Y_{12})$, 
$Y_{10}:= \mbox{Res}_H(Y_{11})$ and  
$Y_{9}:= \mbox{Res}_H(Y_{10})$. 
Notice that $\deg (Y_{12}\cap H) \le \binom{14}{2}$, 
$\deg (Y_{11}\cap H) \le \binom{13}{2}$, and $\deg (Y_{10}\cap H) \le \binom{12}{2}$. 
So by Remark \ref{y} all the vanishings
$\hh^1({\bf{P}}^2,\mathcal {I}_{Y_{12}\cap H}(12)) = 0$,
$\hh^1({\bf{P}}^2,\mathcal {I}_{Y_{11}\cap H}(11)) = 0$ and
$\hh^1({\bf{P}}^2,\mathcal {I}_{Y_{10}\cap H}(10)) = 0$
are satisfied.

Hence by Lemma \ref{castelnuovo}, 
it is sufficient to prove $\hh^1({\bf{P}}^3,\mathcal {I}_{Y_9}(9))=0$.
Recall that for any integer $t \ge 9$ a general union of quadruple, triple and double points 
has good postulation in degree $t$ by \cite{bb,dumnicki2}.  

Thus it is sufficient to prove that $\deg (Y_9)\le \binom{12}{3}$. 
Indeed obviously we have
$\deg(Y_9)\le\deg(X_{13})$.
It is easy to check that
$$\frac32\left(\binom{16}3+13-
\frac{11}{240}\binom{d+3}{3}+\frac3{20}+2(d-13)\right)
\le\binom{12}3$$
for any $d\ge38$.
Hence by \eqref{dacitare} we have $\deg(Y_9)\le\binom{12}3$
and this concludes our proof.
\end{proof}

The crucial tool which allow us to perform our computation in a reasonable time is 
the following special case of \cite[Theorem~1]{dumnicki}.
\begin{theorem}[Dumnicki]
\label{teo-dumnicki}
Let $d,k,m_1,\ldots,m_s,m_{s+1}\ldots,m_r\in 
\NN$.
If
\begin{itemize}
\item $ L_1=\LL(k;m_1,\ldots,m_s)$ is 
non-special;
\item $ L_2=\LL(d;m_{s+1},\ldots,m_r,k+1)$ is 
non-special;
\item $ \vdim L_1=-1$
\end{itemize}
then the system 
$L=\LL(d;m_1,\ldots,m_r)$ is 
non-special.
\end{theorem}
\begin{remark}
To obtain Theorem 
\ref{teo-dumnicki} we have applied 
\cite[Theorem 1]{dumnicki} to 
the case $n=3$ and $\vdim(L_1)=-1$,
since the latter clearly 
guarantees $(\vdim L_1 +1)(\vdim L_2 +1)\geq 0$.
Although this is 
apparently very restrictive, in practice it is very difficult
to find 
different applications which perform 
efficiently.
\end{remark}

The next three lemmas explain how to use Theorem \ref{teo-dumnicki} in order to reduce the computations.

\begin{lemma}\label{primo}
Fix a positive integer $d$ and let $N=\binom{d+3}{3}$. 
For any quadruple of non-negative integers $(w,x,y,z)$, let $Y(w,x,y,z)\subset \mathbb{P}^3$ 
denote a general union of $w$ 5-points, $x$ 4-points, $y$ 3-points and $z$ 
2-points. 
If $Y(w,x,y,z)$ has good postulation in degree $d$ for any quadruple $(w,x,y,z)$ such that
$$0\le z\le 4,$$
$$N-3 \le 35w+20x+10y+4z \le N+\Delta,$$
where $\Delta$ is defined as in \eqref{deltaaa},
then any general quintuple fat point scheme has good postulation in degree $d$.
\end{lemma}
\begin{proof}
Let $Y=Y(w,x,y,z)$ be a general quintuple fat point scheme.
Recall that by Lemma \ref{2agoprop} it is enough to prove the good postulation of $Y$ when
$N-3 \le 35w+20x+10y+4z \le N+\Delta.$
Now assume that $z\geq5$. By Lemma \ref{punti-new} we know that $\LL(3;2^5)$ is non-special and $\vdim(\LL(3;2^5))=-1$.
Then by Theorem \ref{teo-dumnicki} in order to prove that
$Y$ has good postulation in degree $d$, it is enough to prove that $Y(w,x+1,y,z-5)$ has good postulation.
Repeating this step, we reduce to the case when $z\le 4$, and this proves our lemma.
\end{proof}

\begin{lemma}\label{secondo}
Fix a positive integer $d$ and let $N=\binom{d+3}{3}$. 
Given non-negative integers $q,w,x,y,z$, let $Y(w,x,y,z)\subset \mathbb{P}^3$ 
denote a general union of $w$ 5-points, $x$ 4-points, $y$ 3-points and $z$ 
2-points and let $Y'(q,w,x,y,z)$ denote the union of $q$ general $10$-fat points with $Y(w,x,y,z)$.
If $Y'(q,w,x,y,z)$ has good postulation in degree $d$ for any quintuple $(q,w,x,y,z)$ such that
$$0\le z\le 4,\quad 0\le 2x+y\le 21,\quad 0\le w\le 3\ \mbox{ or }\ 0\le x\le3,$$ 
$$N-3 \le 220q+35w+20x+10y+4z \le N+\Delta,$$
where $\Delta$ is defined as in \eqref{deltaaa},
then any general quintuple fat point scheme has good postulation in degree $d$.
\end{lemma}
\begin{proof}
Let $Y=Y(w,x,y,z)$ be a general quintuple fat point scheme.
As in the proof of Lemma \ref{primo} we can assume
$N-3 \le 35w+20x+10y+4z \le N+\Delta$
and $z\le 4$.

Now assume that $2x+y\ge 22$. Then there exist two integers $a,b$ 
such that $a\le x$ and $b\le y$ and $2a+b=22$.
By Lemma \ref{punti-new} we know that the linear system 
$\LL(9;4^a,3^b)$ is non-special and with virtual dimension $-1$. 
So by Theorem \ref{teo-dumnicki} in order to prove that
$Y$ has good postulation in degree $d$, it is enough to prove that $Y'(1,w,x-a,y-b,z)$ has good postulation.
Repeating this step, we reduce to check all the general unions $Y'(q,w,x',y',z)$ such that
$N-3 \le 220q+35w+20x+10y+4z \le N+\Delta$ and $2x'+y'\le21$.

Now assume that $w\ge4$ and $x'\ge4$. 
By Lemma \ref{punti-new} we know that $\LL(9;5^4,4^4)$ is non-special and with virtual dimension $-1$. 
Thus by Theorem \ref{teo-dumnicki} it is enough to prove that $Y'(q+1,w-4,x'-4,y',z)$ has good postulation.
Repeating this step, we complete the proof of the lemma.
\end{proof}

\begin{lemma}\label{terzo}
Fix an integer $d\ge38$ and let $N=\binom{d+3}{3}$.
Given non-negative integers $r,w,x,y,z$, let $Y(w,x,y,z)\subset \mathbb{P}^3$ 
denote a general union of $w$ 5-points, $x$ 4-points, $y$ 3-points and $z$ 
2-points and let $Y''(r,w,x,y,z)$ denote the union of $r$ general $13$-fat points with $Y(w,x,y,z)$.
If $Y''(r,w,x,y,z)$ has good postulation in degree $d$ for any quintuple $(r,w,x,y,z)$ such that
$$0\le z\le4,\quad 0\le w\le 12,\quad 0\le 2x+y\le 41,$$
$$N-3 \le 455r+35w+20x+10y+4z \le N+\Delta,$$
where $\Delta$ is defined as in \eqref{deltaaa},
then any general quintuple fat point scheme has good postulation in degree $d$.
\end{lemma}

\begin{proof}
Let $Y=Y(w,x,y,z)$ be a general quintuple fat point scheme.
As in the proof of Lemma \ref{primo} we can assume
$N-3 \le 35w+20x+10y+4z \le N+\Delta$
and $z\le 4$.

Let $\alpha=\lfloor\frac{2x+y}{42}\rfloor$.
Now if $w \le \alpha -1$, then by Lemma \ref{aritmetico} we also have
$35w\leq\frac1{12}\binom{d+3}3$ and we can apply 
Proposition \ref{pochi5} which says that $Y$ has good postulation.

Assume now that $w \ge \alpha$. 
For $1 \le i \le \alpha$, let $a_i,b_i$ be such that 
$2a_i+b_i = 42$ for all $i$, $\sum _{i=1}^{\alpha }a_i \le x$ and $\sum _{i=1}^{\alpha } b_i \le y$. 
Note that by Lemma \ref{punti-new} all the linear systems 
$\mathcal {L}_3(12;5,4^{a_i},3^{b_i})$ are non-special and with virtual dimension $-1$,
for $1 \le i \le \alpha$.

Then in order to prove that $Y$ has good postulation in degree $d$, 
we apply $\alpha$ times Theorem \ref{teo-dumnicki} and  
we reduce to prove that $Y''(\alpha,w-\alpha,x-\sum a_i,y-b_i,z)$.
So we have to check all the unions of the form
$Y''(r,w',x',y',z)$, where 
$0\le 2x+y\le 41$
and
$N-3 \le 455r+35w'+20x'+10y'+4z \le N+\Delta$.

Now assume that $w'\ge13$ and recall that by Lemma \ref{punti-new} the linear system
$\mathcal {L}_3(12;5^{13})$ is non-special and with virtual dimension $-1$.
Then applying Theorem \ref{teo-dumnicki} we reduce to the case when the number of quintuple points 
is less or equal then $12$, and this completes the proof.
\end{proof}

We are now in position to complete the proof of Theorem \ref{i1}.

\begin{proof}[Proof of  Theorem \ref{i1} for $11\le d\le 52$]\ \\
Let $d$ satisfy $11\le d\le 21$ and let $N=\binom{d+3}{3}$.
Lemma \ref{primo} says that to prove the good postulation of any general union
it is enough to check all the general unions 
with $0\le z\le 4,$ and
$N-3 \le 35w+20x+10y+4z \le N+\Delta$, where $\Delta$ is defined as in \eqref{deltaaa}.
This is precisely Theorem \ref{th11-21}.

Now assume that $22\le d\le 37$. 
By Lemma \ref{secondo} it is enough to prove that a general 
union of $q$ $10$-points, $w$ quintuple points, $x$ quartuple points,
$y$ triple points and $z$ double points has good postulation, when
$0\le z\le 4$, $0\le 2x+y\le 21$, $0\le w\le 3$ or $0\le x\le3$ and
$N-3 \le 220q+35w+20x+10y+4z \le N+\Delta$.
This is Theorem~\ref{th22-37}. 

Finally if $38\leq d\leq 52$, Lemma \ref{terzo} proves that 
it is enough to check all the general unions of
$r$ $13$-points, $w$ quintuple points, $x$ quartuple points,
$y$ triple points and $z$ double points have good postulation, when
$0\le z\le4$, $0\le w\le 12$, $0\le 2x+y\le 41$ and
$N-3 \le 455r+35w+20x+10y+4z \le N+\Delta.$
This is precisely Theorem \ref{th38-52}. 

This concludes the proof of Theorem \ref{i1}.
\end{proof}


\section{A computational proof for the remaining cases}
\label{computer}

In this section we show how several computer calculations
allow to prove Lemma \ref{punti-new}, Theorem \ref{th11-21}, \ref{th22-37},
\ref{th38-52}.

The core of our computation is a programme {\texttt{exact\_case.magma}},
that can be found at 
\begin{verbatim}
http://www.science.unitn.it/~sala/fat_points
\end{verbatim}
We can idealize the operations performed by {\texttt{exact\_case.magma}}
as in the following pseudo-code description of a routine called
{\texttt{exact}}.
\\
\hrule
\begin{center}
{\bf Exact}
\end{center}
\noindent
{\tt Input}\\
$(\texttt{w,x,y,d})$.
\vskip 0.3cm
{\small{
\begin{verbatim}
 N:=Binomial(d+3,3);
 MonomialMatrices(MList);
 L:=35*w+20*x+10*y+4*z; // Length
 // We create the matrix and compute its rank
 BigM:=EvaluationMatrix(MList,q,w,x,y,z);
 r:=Rank(BigM);
 // We check the speciality
 if ((L lt N) and (r ne L)) then
   WriteToFile([q,w,x,y,z]);
 end if;
 if ((L ge N) and (r ne N)) then
   WriteToFile([q,w,x,y,z]);
 end if;
 WriteToFile(certificate);
\end{verbatim}
}}
\hrule
\ \\

The first function 
\texttt{MonomialMatrices} creates a list of
matrices 
$$
  \mbox{\texttt{MList}}=\{M_2,M_3,M_4,M_5\}
$$ 
with monomials entries, where for all matrices the columns correspond 
to all degree-$d$ monomials in four variables, and the rows of $M_h$ 
correspond to the conditions (partial derivatives) of points with multiplicity 
$h$.\\
This list is passed to function \texttt{EvaluationMatrix} 
alongside
with the number of points of given multiplicities.\\ 
The 
function
\texttt{EvaluationMatrix} creates a set of corresponding 
random points with coordinates in the finite field $\FF_p$.
The matrices in \texttt{MList} are evaluated at this set.\\
The output matrix is stored into \texttt{BigM}, whose rank 
is computed
immediately afterwards.\\
Depending on the rank and on the length, if the point configuration is special
then a line is written, otherwise no extra output is needed (see later for
a discussion on the certificate).

Several comments on the above 
algorithm and its implementation are in 
order:
\begin{itemize}
\item The algorithm as described is 
non-deterministic because
      it uses {\em{random}} points; we have 
limited ourselves to use
      pseudorandom sequences and so we need 
to choose a {\em{seed}} (and a {\em{step}})
      whenever we launch an instance of the 
procedure, making the algorithm
      deterministic. In practice, we use the in-built
MAGMA pseudo-random generator: Magma contains an implementation of the 
{\em Monster}
random number generator by G. Marsaglia (\cite{monster}) combined with 
the MD5 hash function. The period of this generator is $2^{29430} - 2^{27382}$ and passes all tests in the Diehard test suite (\cite{DieHard}).

\item The 
bottle-neck of the algorithm is the rank computation.
      Although 
in principle it is possible to check the matrix rank
      over 
$\QQ$, in practice it is much more efficient to perform
      these 
computations over a finite field $\FF_p$, with $p$ a prime.
This is lecit thanks to Remark \ref{wow}.   
The smaller $p$ is, the faster the rank computation is (and the smaller
the memory requirement); however, a smaller prime is more likely
to trigger a wrong rank (\textit{failure}), 
      because of the 
larger number of triggered
      linear relations; therefore, it is important 
to find a prime which
      is both small enough to use a reasonable 
memory amount and large
      enough to avoid failures, if possible.
It turns out that $p=31991$ works well up to the degrees that we needed.
Its size is also very close to $2^{15}$, and so the computer will allocate
exactly $2$ bytes to represent it, without losing an overhead.

\item The rank computation itself is performed by the internal MAGMA rank 
      routine for dense matrices over finite fields. By using several
      optimization techniques, it can compute the ranks also for large matrices
      in a reasonable time. We did some tests and MAGMA's rank routine
      not only outperforms by far any other software package we tried, but
      it also competes with ad-hoc compiled programmed using specialized
      libraries, such as FFLAS-FFPACK (\cite{FFPACK}) or M4RI
      (\cite{M4ri}), although the matrices are not so large as to
      take advantage of sophisticated algorithms such as Strassen's
      (\cite{strassen}) or Winograd's (\cite{winograd}).

 \item 
      The algorithm writes a 
      {\em{digital certificate}}, i.e. a file containing
      vital information 
      enabling a third party to check the correctness 
      of the output. Our certificates vary slightly depending
      on the cases examined, but in each we need: the MAGMA's version,
      the input variables, the pair seed/step,
      the prime, the total computation time and a list of failures
      (if any).\\
      Anyone reading a certificate is able to run the corresponding
      procedure instance and verify the output (assuming that
      our same pseudorandom sequence is utilized).
\end{itemize}

\begin{remark}
\label{wow}
Let $d\ge 11$ be an integer and $p$ be a prime. 
As usual, let $\mathbb {K}$ be any field with characteristic zero. 
Given a quadruple of integers $(w,x,y,z)$, the computer finds 
(in absence of failures) a union $Y(w,x,y,z)\subset \mathbb {P}^3(\mathbb {F}_p)$ 
that is not defective in degree $d$.
By semicontinuity, this proves that a general union
of $w$ 5-points, $x$ 4-points, $y$ 3-points and $z$ 2-points defined over 
$\mathbb {F}_p$ is not defective in degree $d$.
By semicontinuity  this is true for a general union $Y(w,x,y,z)$
defined over $\overline{\mathbb {F}}_p$. 
By semicontinuity this is also true for a general union $Y(w,x,y,z)$ defined 
first over $\overline{\mathbb {Q}}$ and then over $\overline{\mathbb {K}}$. 
Thanks to Lemma \ref{www} this holds also over $\mathbb {K}$.
\end{remark}

The first cases that we checked are the small-degree cases in 
Lemma~\ref{punti-new}. The programme and the digital certificates
can be found at 
\begin{verbatim}
http://www.science.unitn.it/~sala/fat_points/small_cases
\end{verbatim}
Although cases a) and b) of Lemma~\ref{punti-new} were already known
in \cite{dumnicki2}, we redid also them for completeness and check. 

To check the cases in Theorem \ref{th11-21} we prepared a slightly
more complex programme, {\texttt{fat\_points\_brutal.magma}}.
We obviously reuse {\texttt{exact}} but we have to take into
consideration the $\Delta$ values from 
Lemma  \ref{2agoprop}.
A pseudo-code description goes as follows.\\

{\small
\hrule
\begin{center}
{\bf Check of cases 11-21}
\end{center}
\noindent
{\tt Input:}
$\texttt{d}$.
\vskip 0.3cm
{\small{
\begin{verbatim}
N:=Binomial(d+3,3);
// We determine the maximum number of points
z1:=4;y1:=Ceiling(N/10);x1:=Ceiling(N/20);w1:=Ceiling(N/35);
// We set the maximum value of _D, but since the computations are fast
// we leave it except for z>0
_D:=13;
for z in [0..z1] do
 if (z gt 0) then 
  _D:=1;
 end if;
 for y in [0..y1] do
  for x in [0..x1] do
   for w in [1..w1] do // we start from w=1, because w=0 is already in [10]
    L:=35*w+20*x+10*y+4*z; // Length
    if ((L gt N-4) and (L lt N+_D+1)) then
     exact(w,x,y,d);
    end if;
   end for;
  end for;
 end for;
end for;
\end{verbatim}
}}
\hrule
}
\ \\
The programme and the digital certificates can be found at
\begin{verbatim}
http://www.science.unitn.it/~sala/fat_points/11-21/
\end{verbatim}
We report the timings in the following table.
\begin{table}[h!]
\label{tab11-21}
\centering
\tiny{
\caption{Timings in seconds for $d=11\dots 21$ from Theorem \ref{th22-37}}
\begin{tabular}{|c|c|c|c|c|c|c|c|c|c|c|c|}
\hline
${\mathbf{d}}$ & 11    & 12    & 13    & 14  & 15   & 16    & 17   & 18  
               & 19    & 20    & 21 \\    
\hline
               & 54    & 137   & 309   & 683 & 1449 & 2879  & 5736 & 11016 
               & 19857 & 35707 & 61171 \\
\hline
\end{tabular}
}
\end{table}

We proved Theorem \ref{th22-37} similarly, using our programme
{\texttt{fat\_points\_10p.magma}}.
We do not give a pseudo-code, since now it is quite obvious how we proceed.
We note only two key differences. First of all, we used fully the advantage 
offered by the tight determination of $\Delta$. Second, we needed also
10-degree points, but this offered no difficulty, since a slight modification
of {\texttt{exact}} can handle them easily.
The programme and the digital certificates can be found at
\begin{verbatim}
http://www.science.unitn.it/~sala/fat_points/21-37/
\end{verbatim}
We report the timings in the following table. We did also
the defective case $d=21$ as a sanity check.
\newpage
\begin{table}[h!]
\label{tab21-37}
\centering
\tiny{
\caption{Timings in seconds for $d=21\dots 37$ from Theorem \ref{th22-37}}
\begin{tabular}{|c|c|c|c|c|c|c|c|c|c|}
\hline
${\mathbf{d}}$ & 21      & 22 
& 23           & 24            & 25       & 26 & 27 & 28 & \\  
\hline
             & 3539  &    5137 &  7557    &   10911    &  18020     &  20535  & 
29089 & 40221 & \\     
\hline
${\mathbf{d}}$ & 29       & 30   & 31      & 32     & 33     & 34    & 35    & 
36    &  37  \\  
\hline
             & 53583    & 87968 & 107677 & 143758 & 194358 & 255239 & 378412 & 
511234 & 695840 \\
\hline
\end{tabular}
}
\end{table}

\indent
Finally, we proved Theorem \ref{th38-52} in a similar manner, 
by using our programme {\texttt{fat\_points\_13p.magma}}.
Again, a slight modification of {\texttt{exact}} was needed in order
to handle $13$-degree points.
The programme and the digital certificates can be found at
\begin{verbatim}
http://www.science.unitn.it/~sala/fat_points/38-52/
\end{verbatim}
The timings are reported in the following table
\begin{table}[h!]
\label{tab38-52}
\centering
\tiny{
\caption{Timings in seconds for $d=38\dots 52$ from Theorem \ref{th38-52}}
\begin{tabular}{|c|c|c|c|c|c|c|c|c|c|}
\hline
${\mathbf{d}}$ & 38   & 39   & 40   & 41   & 42   & 43   & 44   & 45    \\  
\hline
               &147495&158191&198248&202834&216555&232417&245465&325837 \\     
\hline
${\mathbf{d}}$ & 46   & 47   & 48   & 49   & 50   & 51   & 52    &  \\  
\hline
               &323154&373451&460022&517266&717031&783861&1200723&   \\
\hline
\end{tabular}
}
\end{table}

By observing the timings, we note an exponential behaviour (in $d$)
for Table~1, approximately of behaviour $2^d$.
This is easily explained, because the cost of the rank computation grows as
$d^3$, but the number of cases to be examined grows exponentially.
A similar behaviour can be seen in Table~2, where the
times grow like $(1.4)^d$. Indeed the reason why these latter computations
are feasible lies in the significant cut in the number of cases to
be observed.
However, the real case thinning happens in Table~3,
where the grows is only {\em cubic} in $d$.
This fall from an exponential behaviour to a polynomial one can 
be explained only in a more-or-less constant number of cases to
be considered (the cubic cost being unavoidable due to the cost
of the rank computation). On the other hand, in Theorem \ref{th38-52}
$r$ can take only two values and the other integers are strictly
bounded. As a further check, we computed the number of cases up to
$d=100$ and its maximum value is $405$.

\begin{remark}
We have used 
four Dell servers, each with two processors Intel Xeon X5460 
at 
3.16GHz (for a total of 32 processor cores) and 
with 32 GB's of RAM 
(for a total of 128 GB). 
The underlying operating
system has been Linux, kernel version 
2.6.18-6-amd64.
\end{remark}

\section{The exceptions in degree $9$ and $10$}
\label{d10}

Our main theorem states that a general fat point scheme in $\PP^3$ of
multiplicity $5$ has good postulation in degree $d\ge 11$. Here we
classify all the exceptional cases in degree $10$ and $9$.

Let us consider first the case of degree $10$.
Let $Y$ be a general union of $w$ quintuple points, $x$ quartuple points, $y$
triple points and $z$ double points. Let $N=\binom{13}{3}=286$. 
Then the linear system $L=\LL(10;5^w,4^x,3^y,2^z)$ has virtual dimension 
$\vdim(L)=286-\deg(Y)-1$ where $\deg(Y)=35w+20x+10y+4z$ and the
expected dimension is $\max\{\vdim(L),-1\}$.\\
Our programme checked all the cases with:
\begin{itemize} 
\item either $w\ge 1$ and $286-3\leq \deg(Y)\leq 286+34$, 
\item $w=7,8$ and $\deg(Y)\leq 286+34$.
\end{itemize}
The programme found only nine cases of bad postulation, listed in 
the table below.
In this table, we denote by $e$ the expected dimension 
of the corresponding linear system, by $r$ the rank of the matrix given by our construction, 
and by $d$ the dimension of the linear system.

\begin{table}[h!]
\label{dieci}
\small{
\caption{Exceptions in degree $10$}
\begin{tabular}{|c|c|c|c|c|c|c|c|c|}
\hline    
$w$&$x$&$y$&$z$&$\min(\deg(Y),N)$&$e$&$r$& $d$\\
\hline
9& 0& 0& 0& 286&-1 & 285& 0 \\
\hline
8& 1& 0& 0& 286&-1 & 284& 1\\
8& 0& 1& 1& 286&-1  & 285& 0\\
8& 0& 1& 0& 286&-1 & 283& 2\\
8& 0& 0& 2& 286&-1  & 284& 1\\
8& 0& 0& 1& 284&1  & 282& 3\\
\hline
7& 2& 0& 1& 286&-1  & 284& 1\\
7& 1& 2& 0& 285&0  & 284& 1\\
7& 2& 0& 0& 285&0  & 280& 5\\
\hline
\end{tabular}
}
\end{table}

From this computation we obtain the following classification:
\begin{theorem}
In $\PP^3$ a general union $Y$ of $w$ 5-points, $x$ 4-points, $y$
3-points and $z$ 2-points has good postulation in degree 
$10$, except if the 4-tuple $(w,x,y,z)$ is one of those listed 
in Table \ref{dieci}.
\end{theorem}
\begin{proof}
If $w=0$, then $Y$ is a quartuple general fat point scheme and we already know by \cite{bb,dumnicki2}
that it has good postulation in degree $10$. 
We can thus assume $w>0$.

If $Y$ is a general union of degree 
$283\leq \deg(Y)\leq 320$, our programme checked that there are no other cases of bad postulation, 
except for the ones listed in the table.
 
Now if $Y$ is a general union of degree $\deg(Y)\ge 321$, 
then it contains a subscheme $Y'$ of degree $286\le \deg(Y')\le 320$ which has good postulation,
except if $Y$ is the union of $w\ge10$ quintuple points, 
where the only possible $Y'$ is given by $9$ quintuple points, which has bad postulation.
On the other hand, by our computation we know that the dimension of the linear system 
$\LL(10;5^9)$ is $0$. This means that as soon as we add a general simple point to $Y'$ 
we immediately have an empty linear system. 
This implies that any union of $w\ge10$ quintuple points has good postulation.

Now if $Y$ has degree $\deg(Y)\le 282$, then it is contained in a scheme $Y'$ 
of degree $283\le \deg(Y')\le 286$ which has good postulation, 
obtained by adding only general double points. 
The only case we need to study more carefully are 
$(w,x,y,z)=(8,0,0,0),(7,2,0,0),(8,0,1,0)$, 
which correspond to subschemes of the exceptional cases with $z>0$.
We have checked directly that the first case ($8$ quintuple points) has 
good postulation, 
while the other two are exceptional cases. This completes the proof.
\end{proof}

Some of the exceptional cases we found were already known,  see e.g.\ \cite{laface-ugaglia}
and \cite{volder-laface}. 
Note that all the exceptions we found satisfy the conjecture of Laface-Ugaglia 
(see \cite{laface-ugaglia} and \cite[Conjecture 6.3]{Laface-Ugaglia-standard}).
\\

In the case of degree $9$ we found many more exceptions, 
that we list in the following table.

\pagebreak

\begin{table}
\label{nove}
\small{
\caption{Exceptions in degree $9$}
\begin{tabular}{|c|c|c|c|c|c|c|c|c|}
\hline    
$w$&$x$&$y$&$z$&$\min(\deg(Y),N)$&$e$&$r$& $d$\\
\hline
6& 0& 1& 1& 220&-1& 219&0\\
6& 0& 1& 0& 220&-1 & 216& 3\\
6& 0& 0& 3& 220&-1& 218&1\\
6& 0& 0& 2& 218&1& 214&5\\
6& 0& 0& 1& 214&5 & 210& 9\\
6& 0& 0& 0& 210&9 & 206& 13 \\
\hline
5& 2& 0& 1& 219&0& 217&2\\
5& 2& 0& 0& 215&4& 213&6\\
5& 1& 2& 1& 219&0& 218&1\\
5& 1& 2& 0& 215&4& 214&5\\
5& 1& 1& 3& 217&2& 216&3\\
5& 1& 1& 2& 213&6& 212&7\\
5& 1& 1& 1& 209&10& 208&11\\
5& 1& 1& 0& 205&4& 204&5\\
5& 1& 0& 6& 219&0& 218&1\\
5& 1& 0& 5& 215&4& 214&5\\
5& 1& 0& 4& 211&8& 210&9\\
5& 1& 0& 3& 207&12&206&13\\
5& 1& 0& 2& 203&16& 202&17\\
5& 1& 0& 1& 199&20& 198&21\\
5& 1& 0& 0& 195&24& 194&25\\
\hline
4& 3& 2& 0& 220&-1& 218&1\\
\hline
3& 6& 0& 0& 220&-1& 218&1\\
3& 5& 1& 1& 219&0& 218&1\\
3& 5& 1& 0& 215&4&214&5\\
\hline
\end{tabular}
}
\end{table}

In this case we have tested all the configurations where 
$w\ge1$ and $220-3\leq \deg(Y)\leq 220+34$, 
and all the configurations with $1\le w\le 6$ and $\deg(Y)\leq 220-4$.

From our computational experiments we can deduce the following complete classification:
\begin{theorem}
In $\PP^3$ a general union $Y$ of $w$ 5-points, $x$ 4-points, $y$
3-points and $z$ 2-points has good postulation in degree 
$9$, except if the 4-tuple $(w,x,y,z)$ is one of those listed 
in Table~5.
\end{theorem}
\begin{proof}
If $w=0$, then $Y$ is a quartuple general fat point scheme and it has good postulation in degree $10$
by \cite{bb,dumnicki2}. We can thus assume $w>0$.

If $Y$ is a general union of degree 
$\deg(Y)\leq 254$, our programme checked that there are no other cases of bad postulation, 
except for the ones listed in the table.

Now if $Y$ is a general union of degree $\deg(Y)\ge 255$, 
then it is easy to check that $Y$ contains a subscheme $Y'$ 
of degree $286\le \deg(Y')\le 320$ which has good postulation.
\end{proof}

\begin{remark}
Also in the case of degree $9$ all the exceptions we found satisfy the conjecture of Laface-Ugaglia 
(\cite[Conjecture 6.3]{Laface-Ugaglia-standard}).
\end{remark}

The relevant computations can be found at
\begin{verbatim}
http://www.science.unitn.it/~sala/fat_points/exceptions_9_10/
\end{verbatim}

\section{Further remarks}
\label{fur}
In this final section we provide two remarks on the field characteristics and a direct consequence
of Theorem \ref{i1}.

Since the result by Yang (Remark \ref{y}) is proved only for characteristic zero, we assume in this paper that 
$\mbox{char}({\KK})=0$.
However we underline that our proof of Theorem \ref{i1} can easily be adapted to any
$\mbox{char}({\KK}) \ne 2,3,5$. 
Hence the statement of Theorem \ref{i1} could immediately be generalized to any
characteristic different from $2,3,5$ as soon as 
we know that a general fat point scheme in $\PP^2(\mathbb{F})$ of 
multiplicity $5$ has good postulation in degree $d\ge 3m$, for any field $\mathbb{F}$ with that
characteristic, provided the result holds again for $d=11$ in $\PP^3(\mathbb{F})$.

In positive characteristic the proof of Lemma \ref{2agoprop} fails,
since we cannot make use of Lemma \ref{2ago}. 
However, following the same outline
as in the proof of Lemma \ref{2agoprop} and recalling that a fat point 
always contains a simple point, 
it is not difficult to prove the following lemma.

\begin{lemma}
Let $\FF$ be an infinite field of any characteristic.
Fix an integer $d>0$. 
For any quadruple of non-negative integers $(w,x,y,z)$, let $Y(w,x,y,z)\subset \mathbb{P}^3(\FF)$ 
denote a general union of $w$ 5-points, $x$ 4-points, $y$ 3-points and $z$ 
2-points. 
If $Y(w,x,y,z)$ has good postulation in degree $d$ for any quadruple $(w,x,y,z)$ such that
$$\binom{d+3}{3}-3 \le 35w+20x+10y+4z \le \binom{d+3}{3}+\Delta$$
where 
$$\Delta=\left\{\begin{array}{ll}
14&\mbox{if $w>0$ and $x=y=z=0$,} \\
9&\mbox{if $x>0$ and $y=z=0$,}\\
5&\mbox{if $y>0$ and $z=0$,}\\
2&\mbox{if $z>0$}
\end{array}
\right.
$$
then any general quintuple fat point scheme has good postulation in degree $d$.
\end{lemma}

A straightforward consequence of Theorem \ref{i1} is the following statement, 
whose proof is contained in Remark \ref{proofofcorollary}.
\begin{corollary}\label{i1.0}
Fix non-negative integers $w, x, y, z$ such that
$$35w+20x+10y+4z \ge \binom{14}{3}.$$
Let $Y\subset \mathbb{P}^3$
be a general union of $w$ 5-points, $x$ 4-points, $y$ 3-points and
$z$ 2-points. Then $Y$ has good postulation with respect to any degree.
\end{corollary}

\section*{Acknowledgements}

We would like to thank an anonymous referee for his/her careful inspection of a previous
version of this paper, where he spotted a nasty mistake.
The 
first and second authors were partially supported by MIUR and GNSAGA 
of
 INdAM (Italy).
The third and fourth author acknowledge support 
from the Provincia di Trento's 
grant 
``Metodi algebrici per la 
teoria dei codici correttori e la 
crittografia''.
The authors would like to thank M. Frego for his help in the
computational part.



\providecommand{\bysame}{\leavevmode\hbox 
to3em{\hrulefill}\thinspace}

\end{document}